%% file: BGP-revision.tex
\documentclass[a4paper,11pt,reqno]{amsart}

\usepackage{amssymb,amscd,amsthm,mathrsfs,comment}
\usepackage[all]{xy}         %commutative diagrams
\usepackage[dvips]{graphicx}
\usepackage[ansinew]{inputenc}
\usepackage[centertags]{amsmath}

\newtheorem{theo}{Theorem}[section]
\newtheorem{coro}[theo]{Corollary}
\newtheorem{prop}[theo]{Proposition}
\newtheorem{lemm}[theo]{Lemma}
\newtheorem{add}[theo]{Addendum}
\newtheorem*{clai}{Claim}
\newtheorem{quest}{Question}

\theoremstyle{definition}
\newtheorem{rema}[theo]{Remark}
\newtheorem{defi}[theo]{Definition}
\numberwithin{equation}{section}

   \def\HH{{\mathbb H}}
   
 \def\NN{{\mathbb N}}  
 \def\RR{{\mathbb R}}  \def\TT{{\mathbb T}}
   
 \def\ZZ{{\mathbb Z}}

\def\Si{\Sigma}

\def\Ga{\Gamma}

  \def\cG{\mathcal{G}} \def\cM{\mathcal{M}} 
    
\def\cC{\mathcal{C}}    \def\cU{\mathcal{U}}
    
\def\cE{\mathcal{E}}    
\def\cF{\mathcal{F}}  \def\cL{\mathcal{L}}

\newcommand{\diff}{\operatorname{Diff}}

\newcommand{\en}{\subset}
\newcommand{\eps}{\varepsilon}

\newcommand{\ie}{{\it i.e., } }
\newcommand{\eg}{{\it e.g., } }

\author[C. Bonatti]{Christian Bonatti}
\address{Institut de Math. de Bourgogne CNRS - URM 5584, Universit\'e de Bourgogne Dijon 21004, France}

\author[A. Gogolev]{Andrey Gogolev}
\address{ SUNY Binghamton, N.Y., 13902, U.S.A}

\author[R. Potrie]{Rafael Potrie}
\address{CMAT, Facultad de Ciencias, Universidad de la Rep\'ublica, Uruguay}
\title[Anomalous partially hyperbolic examples II]{Anomalous partially hyperbolic diffeomorphisms II: stably ergodic examples.}
\thanks{A.G. was partially supported by NSF grant DMS-1266282; R.P. was partially supported by CSIC group 618 and the Palis Balzan project.}

\begin{document}
\maketitle

\begin{abstract}
We construct examples of robustly transitive and stably ergodic
partially hyperbolic diffeomorphisms $f$ on  compact $3$-manifolds with fundamental groups of exponential growth
such that $f^n$ is not homotopic to identity for all $n>0$.
These provide counterexamples to a classification conjecture of Pujals.

{ \medskip \noindent \textbf{Keywords:} Partially Hyperbolic
Diffeomorphisms, 3-manifolds, Classification, Robust Transitivity, Stable
Ergodicity.

 }
\end{abstract}

%%%%%%%%%%%%%%%%%%%%%%
%%%%%%%%%%%%%%%%%%%%%%%
\section{Introduction}
%%%%%%%%%%%%%%%%%%%%%%%%
%%%%%%%%%%%%%%%%%%%%%%%%%

Let $M$ be a Riemannian $3$-manifold. In this paper, we say that a diffeomorphism $f\colon M \to M$ is \emph{partially hyperbolic}
if the tangent bundle $TM$ splits into three one-dimensional $Df$-invariant continuous subbundles
$TM= E^{ss}\oplus E^c \oplus E^{uu}$ such that for some $\ell>0$ and for every $x \in M$
$$ \|Df^\ell|_{E^{ss}(x)} \| < \min \{1, \|Df^\ell|_{E^c(x)}\|\} \leq
\max \{1, \|Df^\ell|_{E^c(x)}\| \} < \|Df^\ell|_{E^{uu}(x)} \|.
$$

Sometimes, a stronger notion of \emph{absolute partial hyperbolicity} is used.
This means that $f$ is partially hyperbolic and there exists $\lambda < 1 < \mu$ such that

$$ \|Df^\ell|_{E^{ss}(x)} \| < \lambda<  \|Df^\ell|_{E^c(x)}\| < \mu < \|Df^\ell|_{E^{uu}(x)}
\| $$
\medskip
The subbundles $E^{ss}$, $E^c$ and $E^{uu}$ depend on $f$ and we will indicate this, when needed, using a subscript, \eg $E^{ss}_f$.

In 2001, it was informally conjectured in a talk by E. Pujals\footnote{The conjecture was formalized in~\cite{BW},
see also~\cite[Section 2.2]{HasPes} and~\cite{CHHU}.}  that all examples of transitive partially hyperbolic diffeomorphisms,
up to taking finite lifts and iterates, fall into one of the following classes:
\renewcommand\labelenumi{\theenumi.}

\begin{enumerate}
 \item deformations of linear Anosov automorphisms of the 3-torus $\TT^3$;
 \item deformations of skew products over a linear Anosov automorphism of the 2-torus $\TT^2$;
 \item deformations of time-one maps of Anosov flows.
\end{enumerate}

In this paper we provide counterexamples to Pujals' conjecture.

\begin{theo}\label{t.main} There exist a closed orientable $3$-manifold $M$ and an
absolutely partially
hyperbolic diffeomorphism $f\colon M\to M$ which satisfies the following properties
\begin{itemize}
\item $M$ admits an Anosov flow;
\item $f^n$ is not homotopic to the identity map for all $n>0$;
\item $f$ is volume preserving;
\item $f$ is robustly transitive and stably ergodic.
\end{itemize}
\end{theo}

Note that, because $M$ admits an Anosov flow, the fundamental group of $M$ has exponential growth.
Therefore $M$ and its finite covers do not admit Anosov automorphisms and partially hyperbolic skew products.
Further, because iterates of $f$ are not homotopic to identity, it follows that $f$, its iterates, and their finite lifts
are not homotopic to the time-one map of an Anosov flow. We conclude that the diffeomorphism $f$ given by Theorem~\ref{t.main},
indeed, gives a counterexample to the Pujals' conjecture.

We present two classes of examples, both of which yield the statement of Theorem~\ref{t.main}
\begin{itemize}
 \item one on the unit tangent bundle of a surface of genus two or higher;
 \item the second class is based on a transitive Anosov flow which admits a transverse torus disjoint with a periodic orbit.
\end{itemize}
\begin{rema}
 It is in fact plausible that the latter construction can be applied to any transitive Anosov flow with a transverse torus.
 However for the sake of simplicity and clarity, we will only present here an example based on
 the specific Anosov flow constructed in~\cite{BL}.
\end{rema}
 %%\marginpar{rewrite if the general family works out.}

%%%%%%%%%%%%%%%%%%%%%%%%%%%%%%%%%%%%%%%%%
\subsection{Overview of the constructions}
%%%%%%%%%%%%%%%%%%%%%%%%%%%%%%%%%%%%%%%%%%%
In both constructions we start with a Riemannian manifold with an Anosov flow and then perform a deformation (by changing the Riemannian metric in the first case and by considering finite lifts in the second case) which preserves the strength of the partial hyperbolicity of the Anosov flow. Both manifolds admit an incompressible torus (in the first case it contains two periodic orbits and in the second case it is transverse to the flow) and in both cases we consider a Dehn twist in a neighborhood of this torus which preserves the partially hyperbolic structure provided the deformation is sufficiently large. Let us make a more detailed outline.

The first construction is based on the time-one map $f\colon T^1S\to T^1S$ of the geodesic flow on a hyperbolic surface $S$.
We fix a simple closed geodesic $\gamma$ and consider a Dehn twist $\rho$ along $\gamma$.
Its differential induces diffeomorphism $D\rho\colon T^1S\to T^1S$.
To find a partially hyperbolic diffeomorphism in the mapping class of $D\rho\circ f$
we deform the hyperbolic metric within the space of hyperbolic metrics in such
a way that the length of $\gamma$ goes to zero. Time-one map $f$ deforms accordingly and stays partially hyperbolic.
Because the ``collar" of $\gamma$ becomes a very thin tube, it is possible to deform the Dehn twist $\rho$ so that it
becomes an ``almost isometry" of the surface. Hence, taking the composition $D\rho\circ f$ does not destroy partial
hyperbolicity of $f$. It is possible to adjust this perturbation in order to make it volume preserving and still have partial hyperbolicity. We will also point out that the same constructions works starting with the time-one map of Handel-Thurston Anosov flow.

Our second construction is an adaptation of the one in~\cite{BPP}.
We start with a conservative Anosov flow transverse to a torus $T$ (see~\cite{BL}).
We compose the time-$N$ map of the flow, for some large $N$, with a Dehn twist along the torus, supported on
a fundamental domain. In the non-transitive case considered in \cite{BPP} the unstable and strong unstable (resp. stable and strong stable) foliations
were kept unchanged in the negative (resp. positive) iterates of the fundamental domain; then the proof of partial hyperbolicity relied on the fact that, for $N$ large enough,
the Dehn twist preserves the transversality of the foliations. In the transitive case, $N$ cannot be chosen larger than the smallest return time on the torus;
moreover, none of the foliations are kept unchanged. However both difficulties would resolve if we could increase the return time without changing
the dynamics, in particular the strength of the partial hyperbolicity.  In the current paper we do it by considering the lift of the Anosov flow on a sufficiently large finite cyclic cover of the original manifold. 

It is possible to construct examples without considering finite lifts by using a different mechanism, not preserving the strength of the partial hyperbolicity. This requires a different approach and will be delegated to a future paper.

%%%%%%%%%%%%%%%%%%%%%%%%%%%%%%%%%%%%%%%%%%%
\subsection{Pujals' conjecture revisited}
%%%%%%%%%%%%%%%%%%%%%%%%%%%%%%%%%%%%%%%

A partially hyperbolic diffeomorphism $f$ is called \emph{dynamically coherent} if the subbundles $E^{ss}\oplus E^c$ and   $E^c \oplus E^{uu}$
are tangent to  invariant $2$-dimensional foliations, denoted by $W^{cs}_f$ and $W^{cu}_f$, respectively. Then, these foliations
intersect along an invariant $1$-dimensional foliation $W^c_f$ tangent to $E^c$. Under some technical assumptions, the pair $(f,W^c_f)$
is known to be structurally stable~\cite{HPS}:
for every diffeomorphism $g$, $C^1$-close to $f$, there is a homeomorphism $h$ conjugating the foliations $W^c_f$ and $W^c_g$ and the points
$h^{-1}gh(x)$ and $f(x)$  are uniformly bounded distance apart in the center leaf $W^c_f(x)$. We say that $h$ is a \emph{$W^c$-conjugacy}.

An example of non-dynamically coherent partially hyperbolic diffeomorphism on the torus $\mathbf{T}^3$ has been built by \cite{HHU-noncoherent}.
This example is not transitive and not absolutely partially hyperbolic.

Pujals' conjecture admits a stronger formulation, for dynamically coherent partially hyperbolic diffeomorphisms.
It asserts that all such diffeomorphisms must be $W^c$-conjugated  (up to finite iterates and lifts to finite covers) to one of the three models.
Details can be found in~\cite{CHHU}.

\begin{rema}
Counterexamples to this strong version of Pujals' conjecture were given recently in~\cite{BPP},
though the examples are not transitive.
We do not know if the transitive examples presented here are,
in fact, dynamically coherent (see subsection \ref{s.Questions}).
\end{rema}

\begin{rema}
Several positive classification results were established in~\cite{BW,BBI,BI,Par,HP1,HP} and certain families
of 3-manifolds are now known only to admit partially hyperbolic diffeomorphisms which are on Pujals' list.
\end{rema}

In our view, what Pujals was proposing is that it could be possible to reduce the classification of partially hyperbolic diffeomorphisms in dimension $3$
to the classification of Anosov flows (even though the latter are far from being classified).
The new examples greatly enrich the partially hyperbolic zoo in dimension 3.
Still, the program of reducing the classification to that of Anosov flows, should not be abandoned.

The following questions arise naturally:

\begin{quest}
Assume that a manifold $M$ with exponential growth of fundamental group admits a partially hyperbolic diffeomorphism. Does it also admit an Anosov flow?
\end{quest}

Notice that the main known obstruction for a manifold to admit an Anosov flow is the non-existence of Reebless foliations.
This is also an obstruction for the existence of partially hyperbolic diffeomorphisms~\cite{BI}.

\begin{quest}
Let $f\colon M \to M$ be a (dynamically coherent) partially hyperbolic diffeomorphism homotopic to the identity,
is it $W^c$-conjugate to an Anosov flow\footnote{As remarked in \cite{BW} it might be better to consider $W^c$-conjugacy
to a topologically Anosov flow for technical reasons. See \cite[Conjecture 1]{BW}.}?
\end{quest}

The counterpart of the above question on \emph{small manifolds}
(\ie with fundamental group of polynomial growth) has been addressed and admits a complete answer~\cite{HP1}.
Among manifolds with fundamental group of exponential growth, only the case of solvable fundamental group is known~\cite{HP}.

In fact, a natural (albeit somewhat vague) question which arises in view of our examples is the following:

\begin{quest}
Let $f\colon M \to M$ be a partially hyperbolic diffeomorphism on a 3-manifold $M$ admitting an Anosov flow.
What is the relationship between $f$ and this flow?

For instance, does $M$ admit an Anosov flow  $X$ so that both $f$ and $X$ leave positively (resp. negatively)
invariant the same strong unstable (resp. stable) cone-field ? Is the center-foliation of $f$ (if it exists) equivalent to the center foliation of a topologically Anosov flow? 
\end{quest}

 Note that a 3-manifold may admit many Anosov flows which are not topologically orbit equivalent (see \cite{BBY} and references therein) and so the answer to the previous question may depend on the Anosov flow on $M$.

\subsection{Further properties and questions}\label{s.Questions}

New examples are the source of new questions, but also a motivation to look again at previous ones. For example, in Section \ref{s.saddlenode} we show that one of our examples possesses\footnote{It is plausible that all our new examples have this property.} periodic center leaves with new type of dynamical behavior which was not present in previous examples (see \cite[Section 7.3]{BDV}). Here by \emph{periodic center leaf} we mean a complete curve tangent to $E^c$ invariant by some power of $f$. 

Other questions that must be tackled in view of the new examples pertain their dynamics, ergodicity in the volume preserving case, etc (see for example \cite{CHHU,WilkinsonSurvey}). Let us formulate some questions which we believe to be interesting:

\begin{quest} Are the new examples presented in this paper dynamically coherent?
\end{quest}

Since the examples are constructed by composing with a perturbation (with relatively large support) we do not have much control on the dynamics or structure of the bundles after perturbation (we establish partial hyperbolicity using cone-field criteria). Notice that thanks to some criteria introduced in \cite{BW} it is enough to show that the examples have a \emph{complete} center-stable manifold (meaning that the saturation of a center-stable leaf by strong stable ones is complete in the metric induced by the Riemannian metric of the manifold). If one shows dynamical coherence, it seems natural to test other properties too:

\begin{quest}
Are the examples plaque-expansive?
\end{quest}
(See \cite[Chapter 7]{HPS} for definitions.) It is natural to study the $W^c$-conjugacy type of the examples, what is the dynamics of the center leaves, etc. Also note that all previously known examples have the property that they admit models with \emph{smooth} center foliations.\footnote{\cite{FG} presents special examples of a higher dimensional dynamically coherent partially hyperbolic diffeomorphisms with non-smooth center foliation. It is plausible that these examples cannot be homotoped to ones with smooth center foliation. However they are $W^c$ conjugate to  partially hyperbolic diffeomorphisms with smooth center foliation.} It is unlikely that this will be the case for our examples, still we pose this as a question:
\begin{quest}
Is it possible to homotope any of the examples of the current paper to a partially hyperbolic diffeomorphism with a smooth center foliation?
\end{quest}

\subsection{Organization of the paper}

As we have already mentioned, the paper contains two families of examples which provide a proof of Theorem~\ref{t.main}. The presentation of each of the examples is independent and can be read in any order. The only exception is the proof of robust transitivity and stable ergodicity which is the same proof and is carried out in Subsection~\ref{sect:RTSE-T1S}.

In Section~\ref{s.Exampleunittangent} we present the example on the unit tangent bundle of a hyperbolic surface and a related example on a certain graph manifold. In Section \ref{s.ExampleBL} we present the example starting with a transitive Anosov flow transverse to a torus which is not a suspension. Also, in Section~\ref{s.saddlenode} we discuss properties of periodic center leaves for the latter example.

Finally, we made an effort to make this paper (topologically) self-contained and gave elementary proofs of certain known results on 3-manifolds relying on explicit description of the 3-manifolds at hand, see Section~\ref{s.homology} and also Remark~\ref{rema_HT_on_homology}.

%%%%%%%%%%%%%%%%%%%%%%%%%%%%%%%%%%%%%%%%%%%%%%%%%%%%%%%%%%%%%%%%%%%%%%%%%%%%
%%%%%%%%%%%%%%%%%%%%%%%%%%%%%%%%%%%%%%%%%%%%%%%%%%%%%%%%%%%%%%%%%%%%%%%%%%%%%
\section{An example on the unit tangent bundle of a surface}\label{s.Exampleunittangent}
%%%%%%%%%%%%%%%%%%%%%%%%%%%%%%%%%%%%%%%%%%%%%%%%%%%%%%%%%%%%%%%%%%%%%%%%%%%%%%%
%%%%%%%%%%%%%%%%%%%%%%%%%%%%%%%%%%%%%%%%%%%%%%%%%%%%%%%%%%%%%%%%%%%%%%%%%%%%%%%%

\subsection{A sequence of hyperbolic surfaces}

Let $S$ be an orientable closed surface of genus 2 or higher. A Riemannian metric on $S$ is hyperbolic if the curvature is constant $-1$.
Any closed surface $S$ endowed with a hyperbolic metric $g$ is called a hyperbolic surface. Any closed hyperbolic surface
$(S,g)$ has the hyperbolic plane $\HH^2$ as its universal cover.  In other words, $(S,g)$ is isometric
to the quotient of $\HH^2$ by a discrete co-compact subgroup of isometries of $\HH^2$  (Fuchsian group) acting freely on $\HH^2$; we denote by
$\Pi_g\colon\HH^2\to (S,g)$ this universal cover, which is a local isometry.
The group of deck transformations of the cover $\Pi_g$ is identified with
the fundamental group $\pi_1(S)$.

A marked surface is a surface endowed with a set of generators of $\pi_1(S)$.
The space $\textup{Teich}(S)$ of equivalence classes of marked constant $-1$ curvature Riemannian metrics on $S$
is called the Teichm\"uller space, and we refer to~\cite{FLP} for its properties.

Let $\gamma$ be an essential (\ie  $[\gamma]\neq 0$ in $\pi_1(S)$) simple closed curve in $S$.
 It is easy to see that there exists a sequence of hyperbolic metrics
$\{[g_n]\in\textup{Teich}(S); n\ge 1\}$ such that $\gamma$ is a geodesic for $g_n$ and its  length $\ell_n$ with respect to $g_n$  monotonically decreases to $0$ as $n\to\infty$ .
This can be seen, for instance, by using Fenchel-Nielsen coordinates on $\textup{Teich}(S)$. More precisely, there is a decomposition of the surface $S$ in
\emph{pair of pants}  (a \emph{pair of pants} is topologically the sphere $S^2$ minus the interior of $3$ disjoint discs), so that $\gamma$ is
one of the boundary component of a pair of pants. Then (see \cite{FLP}) the length of the boundary components of
the pair of pants  (and hence the length of $\gamma$)
can be chosen arbitrary.

We fix such a sequence of hyperbolic metrics $\{g_n; n\ge 1\}$.
We denote $\Pi_n=\Pi_{g_n}\colon\HH^2\to (S,g_n)$ the corresponding universal covers.

\subsection{Geodesic flows}

Let $TS$ be the tangent bundle of the surface $S$. We denote by $T^1S$ the \emph{unit tangent bundle} of $S$:
given a Riemannian metric $g$ on $S$, the unit tangent bundle is the level set
$$
\{v\in TS: \|v\|_g=1\}.
$$
It is a smooth circle bundle over $S$. As we will endow $S$ with a family of metrics, we can also define the unit tangent bundle
without using a specific metric:  $T^1S$ is the circle bundle over $S$ , defined as being the the quotient of
the bundle $TS\backslash S$, (where $S$ is the zero section) by identifying
$v\in T_x S\setminus \{0_x\}$ and $u\in T_xS\setminus\{0_x\}$ if and only if $v=cu$ for some $c>0$.

Note that, for hyperbolic metrics $g_n$ introduced earlier,
the geodesic flow on the tangent bundle $TS$ restricts to the level set $\{v: \|v\|_{g_n}=1\}$.
Hence, via the above canonical identifications of all unit tangent bundles with $T^1S$,
each metric $g_n$ gives rise to its  geodesic flow $\cG_n$ on $T^1S$.

Let $f_n\colon T^1S\to T^1S$ be the time-one map of the geodesic flow $\cG_n$ of $g_n$,
$n\ge 1$, and let $f\colon T^1\mathbb H^2\to T^1\mathbb H^2$ be the time-one map of the geodesic
flow on the hyperbolic plane $\HH^2$.

 Then, because each $(S,g_n)$ is covered by $\HH^2$, we have the following commutative diagram

\begin{equation}\label{eq_covering0}
\xymatrix{
 T^1\HH^2 \ar[r]^f\ar[d] &  T^1\HH^2 \ar[d]\\
  T^1S \ar[r]^{ f_n} & T^1S
}
\end{equation}
where the vertical arrows are induced by the derivative of the covering map $\Pi_n$.

The hyperbolic metric $g$ on $\HH^2$ (resp. $g_n$ on $S$) induces the Sasaki metric $\hat g$ on $T^1\HH^2$ (resp. $\hat g_n$ on $T^1S$).
We will only use the following properties of the Sasaki metrics:
\begin{itemize}
\item The metric $\hat g$ is invariant under derivatives of the isometries of $\HH^2$,
 \item The derivative of the projection $\Pi_n$ is a local isometry from $(T^1\HH^2,\hat g)$ to $(T^1S,\hat g_n)$.
 \item The Anosov splitting for the geodesic flow on $T^1\HH^2$ is orthogonal with respect to the Sasaki metric.
\end{itemize}

Thus we have the following commutative diagram
\begin{equation}\label{eq_covering}
\xymatrix{
 (T^1\HH^2, \hat g) \ar[r]^f\ar[d] &  (T^1\HH^2, \hat g)\ar[d]\\
  (T^1S, \hat g_n)\ar[r]^{ f_n} & (T^1S, \hat g_n)
}
\end{equation}
Here the vertical arrows are local isometries. This observation will be crucial for the proof of partial hyperbolicity of
the example which we are about to construct.

%%%%%%%%%%%%%%%%%%%%%%%%%%%%%%%%%%%%%%%%%%%%%%%%%%%%%%%%%%%%%%%%%%%%%
\subsection{Collar neighborhoods of the geodesic $\gamma$ for $g_n$}
%%%%%%%%%%%%%%%%%%%%%%%%%%%%%%%%%%%%%%%%%%%%%%%%%%%%%%%%%%%%%%%%%%%%

Denote $S^1$ the circle $\RR/\ZZ$ and  by $C_n$, $n\ge 1$, the cylinder $[0,1]\times S^1$ equipped with hyperbolic metric
$$
dx^2+ \ell_n^2\cosh^2(x)dy^2,  (x,y)\in [0,1]\times S^1.
$$
One can visualize $C_n$ as follows: given an oriented geodesic $\sigma$ of $\HH^2$ there is a unique $1$-parameter group $h_{\sigma,t}$ of isometries preserving $\sigma$.  These isometries are called
\emph{translations of axis $\sigma$} and $h_{\sigma,t}$ acts on $\sigma$ as a translation of length $t$. For $t\neq 0$ $h_{\sigma,t}$ acts freely and properly on $\HH^2$:
thus  the quotient space $\HH^2/h_t$, for $t>0$, is a cylinder $\Si_t$ (diffeomorphic to $\RR\times S^1$) and, as $h_{\sigma,t}$ is an isometry, the hyperbolic metric $g$ goes down on an
hyperbolic metric $g^t$ on $\Si_t$. The cylinder $\Ga_t$ supports a unique closed geodesic $\sigma_t$, which is the projection of $\sigma$; its length is $t$.
We denote by $\Ga_{t,+}$ and $\Ga_{t,-}$ the two closed half cylinders obtained by cutting $\Ga_t$ along $\sigma_t$.  We denote by  $C_{t,1}$ the subset of $\Ga_{t,+}$ of point
whose (hyperbolic) distance from $\sigma_t$ is less than or equal to~$1$. This provides an equivalent description

$$C_n \quad \mbox{is isometric to}\quad  C_{l_n,1}$$

\medskip

Let us denote by $\gamma_n$ the curve $\gamma$ considered as a geodesic of $(S,g_n)$.
If we choose an orientation on the ``left" boundary component of $C_n$ and on the geodesic $\gamma_n$
then there exists a unique (up to the obvious $S^1$ action on $C_n$) locally isometric immersion
$$
\varphi_n\colon C_n\to (S,g_n),
$$
which sends the ``left" boundary component to $\gamma_n$ in orientation preserving manner as shown on Figure~\ref{fig-neck}.

Clearly we can assume that $\ell_n<2\sinh^{-1}(1/\sinh(1))$ for all $n\ge 1$.
By the Collar Lemma~\cite[p. 402]{FM}, this condition implies that $\varphi_n$, $n\ge1$, are, in fact, isometric embeddings.
From now on we will identify the cylinder $C_n$ with its image under $\varphi_n$ and refer to it as the {\it neck}.

\begin{figure}[htbp]\begin{center}
\input{fig-longneck.pstex_t}
\caption{}\label{fig-neck}
\end{center}\end{figure}

%\begin{figure}[htbp]
%\begin{center}
%\includegraphics{fig-longneck.eps}
%\end{center}
% \caption{}
%\label{fig-neck}
%\end{figure}

\subsection{Dehn twists}
Now define the Dehn twist $\rho_n\colon (S, g_n)\to (S, g_n)$ by
\begin{equation}
 \label{eq_rho}
\rho_n(p)=
\begin{cases}
p, \mbox{if}\,\, p\notin C_n \\
(x, y+\rho(x)), \mbox{if}\,\, p=(x,y)\in C_n
\end{cases}
\end{equation}
Here $\rho\colon[0,1]\to S^1$ is a $C^\infty$ ``twist function''; \ie $\rho$ has the following properties
\begin{itemize}
\item $\rho$ is $C^\infty$ flat at $0$ and at $1$;
\item $\rho(0)=\rho(1)$;
\item $\rho$ is increasing (we picked an orientation on $S^1$).
\end{itemize}

Next Proposition~\ref{prop_close_to_id} asserts that the Dehn twists tend to the identity maps in the $C^\infty$ topology as $n$ tends to infinity.
This statement does not have an obvious meaning, as each diffeomorphism $\rho_n$ is considered on $S$ endowed with the metric $g_n$:
thus the $C^\infty$ distance we consider on $S$ depends
  on $n$.
  The idea is to consider lifts on $\HH^2$ where the metric is fixed. Let us explain that precisely.

 \begin{defi} Consider  the upper half plane $\HH^2\subset \RR^2$. Let $h_n$ be a sequence of diffeomorphisms on $\HH^2$ and let $X_n\subset \HH^2$ be a sequence of sets. We say that
 $d_{C^\infty}(h_n|_{X_n},id)\to 0$ if for any $\varepsilon>0$ and any $m>0$ there is $n_0$ so that, for any $n\geq n_0$  we have the following property:\\
 for any $x\in X_n$ there is an isometry $\phi_x$ of $\HH^2$ so that $\phi_x(x)=(0,1)$ and so that all
 the derivatives of $\phi_x\circ h_n\circ\phi_x^{-1}$ of order less than
 or equal to $m$
 at the point $(0,1)$ are less than $\varepsilon$.

 \end{defi}

\begin{defi} Let $h_n$ be a sequence of diffeomorphisms on $S$ supported on the neck $C_n$, $n>0$ (that is, $h_n(x)=x$ for $x\notin C_n$).
Let $\tilde \gamma_n$ and $\tilde C_n$
be the lifts
of $\gamma_n$ and $C_n$ to $\HH^2$ so that $\tilde \gamma$ is a boundary component of the strip $\tilde C_n$.
Let $\tilde h_n$ be a lift of $h_n$ on $\HH^2$ such that $\tilde h_n$ is the identity on $\tilde \gamma_n$.

We say that the diffeomorphisms $h_n$ tend to the identity with respect to the $C^\infty$ distance on $(S,g_n)$,
and we write $d_{C^\infty,n}(h_n,id)\to 0$ if the restrictions of
$\tilde h_n$ to $\tilde C_n$ tends to the identity map in the $C^\infty$ topology.
\end{defi}

 We are now ready to state the key property of Dehn twists.

\begin{prop}\label{prop_close_to_id}
The sequence of Dehn twists $\rho_n\colon (S, g_n)\to (S, g_n)$ constructed above has the following property
$$
d_{C^\infty,n}(\rho_n,id_S)\to 0, n\to\infty.
$$
\end{prop}

\begin{proof}
Clearly we only need to pay attention to the neck $C_n\subset S$.
After rescaling the coordinates $(x,y)\mapsto(x, \ell_n y)=(\bar x,\bar y)$ the expression for $g_n$ becomes independent of $n$
\begin{equation}
\label{eq_hyp_metric}
g_n=d\bar x^2+\cosh^2(\bar x) d\bar y^2;
\end{equation}
The formula for $\rho_n$ becomes
$$
\rho_n(\bar x,\bar y)=(\bar x,\bar y+\ell_n\rho(\bar x)).
$$
The proposition follows because $\rho'$ and it's higher derivatives are uniformly bounded and $\ell_n\to 0$ as $n\to\infty$. (Here we slightly abused notation by writing $\ell_n\rho$ for the same ``twist function'' $[0,1]\to(S^1,d\bar y)$ into the rescaled circle.)
\end{proof}
The following is an immediate corollary.
\begin{coro}\label{cor_almost_isometry}
Let $\rho_n^*g_n$ denote the pull-back metric of $g_n$ using $\rho_n$. Then
$$
d_{C^\infty}(\rho_n^*g_n,g_n)\to 0, n\to\infty.
$$
\end{coro}

Let $D\rho_n\colon TS\to TS$, $n\ge1$, be the differential map. We abuse notation and also write $D\rho_n\colon T^1S\to T^1S$ for the induced diffeomorphism of $T^1S$ given by $[v]\mapsto [D\rho_n(v)]$.
Because Sasaki construction only uses the first derivatives of the metric we also have the following.

\begin{coro}\label{cor_close_to_id}
The sequence of Dehn twists $\rho_n\colon (S, g_n)\to (S, g_n)$ constructed above has the following property
$$
d_{C^\infty}(D \rho_n,id_{T^1S})\to 0, n\to\infty.
$$
\end{coro}

\begin{coro}
\label{cor_almost_isometry2}
For Sasaki metrics $\hat g_n$, $n\ge 1$, we have
$$
d_{C^\infty}((D\rho_n)^*\hat g_n,\hat g_n)\to 0, n\to\infty.
$$
\end{coro}

%%%%%%%%%%%%%%%%%%%%%%%%%%%%%%%%%%%%%%%%%%%%%%%%%%%%%%%%%%%%%%%%%%%%%%%%%%%%%%%%%%%%%%
\subsection{Large perturbations preserving a partially hyperbolic structure}
%%%%%%%%%%%%%%%%%%%%%%%%%%%%%%%%%%%%%%%%%%%%%%%%%%%%%%%%%%%%%%%%%%%%%%%%%%%%%%%%%%%%%
We now introduce a new definition in order to understand under what condition
a large perturbation can preserve the (absolute) partially hyperbolic structure of a diffeomorphism.

 Let $(M,g)$ be a (not necessarily compact) complete Riemannian 3-manifold let $f\colon M\to M$
 be an absolutely partially hyperbolic diffeomorphism for which the splitting $E^{ss}\oplus E^c\oplus E^{uu}$ is orthogonal and the following inequalities hold,
\begin{equation}
\label{eq_ph}
 \|Df|_{E^{ss}(x)} \|_g < \lambda<\lambda'<  \|Df|_{E^c(x)}\|_g <\mu'< \mu < \|Df|_{E^{uu}(x)}\|_g
\end{equation}
where $\lambda'<\lambda<1<\mu'<\mu$ are constants. We say that a sequence of diffeomorphisms $\{h_n\colon M\to M; n\ge 1\}$ is {\it ph-respectful} relative to $(f,g)$ if
\begin{equation}\label{eq_respectful1}
\sup_{x\in M}\angle_g(Dh_nE_f^\sigma(x), E_f^\sigma(h_n(x))\to 0, n\to \infty, \,\,\,\sigma=ss, c, uu
\end{equation}
and
\begin{equation}\label{eq_respectful2}
d_{C^0}(h_n^*g,g)\to 0, n\to\infty
\end{equation}
Note that Equation~(\ref{eq_respectful2})
implies that both $\|Dh_n(x)\|$ and $\cM(Dh_n(x))= \|Dh_n^{-1}(h(x))\|^{-1}$ uniformly converge to $1$ as $n\to +\infty$.

\begin{prop}\label{prop_ph}
Let $(M ,g)$, $f$ and a $ph$-respectful sequence $\{h_n\colon M\to M; n\ge 1\}$ be as above. Then for all sufficiently large $n$ the diffeomorphism $h_n\circ f$ is absolutely partially hyperbolic with respect to $g$.
\end{prop}

\begin{proof}[Proof of Proposition~\ref{prop_ph}]
For $\alpha\in(0,\pi/2)$ define the following unstable cone field on $(M, g)$
$$
\cC_\alpha^{uu}=\{v\in T_xM: \angle (v, E^{uu})<\alpha\}.
$$
Then, by assumption~(\ref{eq_ph}) there exists $\alpha\in(0,\pi/4)$ such that
\begin{equation}\label{eq_cone_invariance}
Df(\cC_{\pi/4}^{uu})\subset\cC_\alpha^{uu}\subset \cC_{\pi/4}^{uu}
\end{equation}
Now pick any non-zero vector $v\in Df(\cC_{\pi/4}^{uu})$. Using~(\ref{eq_respectful2}) and~(\ref{eq_respectful1}), we have
\begin{multline*}
\angle(Dh_nv, E^{uu})\le
\\ \angle(Dh_nv, Dh_n(E^{uu}))+\angle(Dh_n(E^{uu}),E^{uu})\to\angle(v, E^{uu}), n\to\infty
\end{multline*}
uniformly in $v\in Df(\cC_{\pi/4}^{uu})$. It follows that for all sufficiently large $n$
$$
Dh_n(\cC_\alpha^{uu})\subset\cC_{\pi/4}^{uu}
$$
and, by combining with~(\ref{eq_cone_invariance}) we obtain
$$
D(h_n\circ f)(\cC_{\pi/4}^{uu})\subset\cC_{\pi/4}^{uu}.
$$
Also using~(\ref{eq_ph}) and~(\ref{eq_respectful2}) on can check that there exist constants $\nu>\nu'>1$ such that for all sufficiently large $n$
$$
\|D(h_n\circ f)v\|_g\ge\nu\|v\|_g\;\;\;\mbox{if}\;\;v\in Df(\cC_{\pi/4}^{uu})
$$
and
$$
\|D(h_n\circ f)v\|_g\le\nu'\|v\|_g\;\;\;\mbox{if}\;\;v\notin \cC_{\pi/4}^{uu}.
$$
Because we assume that the partially hyperbolic splitting of $f$ is orthogonal,  we can check that $Df$ satisfies such inequalities. For large $n$,
diffeomorphism $h_n$ almost does not affect the
norms of vectors and we obtain the posited inequalities for $D(h_n\circ f)$.

By reversing the time we can obtain analogous properties of the (analogously defined)
stable cone field $\cC_{\pi/4}^{ss}$ hold with respect to $(h_n\circ f)^{-1}$.
It is well-known~(see, \eg~\cite{HasPes}) that existence of such cone fields imply absolute partial hyperbolicity.
\end{proof}

%%%%%%%%%%%%%%%%%%%%%%%%%%%%%%%%%%%%%%%%%%%%%%%%%%%%%%%%%%%%%%%%%%%%%%%%%%%%%%%%%%%%%%%%%
\subsection{The basic example}  We are ready to present the basic version of our example.
%%%%%%%%%%%%%%%%%%%%%%%%%%%%%%%%%%%%%%%%%%%%%%%%%%%%%%%%%%%%%%%%%%%%%%%%%%%%%%%%%%%%%%%%%%

\begin{theo} \label{theorem_basic}
Let $S$ be a surface of genus 2 or higher, let $f_n\colon T^1S\to T^1S$, $n\ge 1$ be the time-one maps of
the geodesic flows and let $D\rho_n\colon T^1S\to T^1S$, $n\ge 1$, be the diffeomorphism induced by the Dehn
twists $\rho_n$ as described above. Then for all sufficiently large $n$ the diffeomorphisms $D\rho_n\circ f_n$ are
absolutely partially hyperbolic. Furthermore, these diffeomorphisms and their finite iterates are not homotopic to identity.
\end{theo}

Recall that the Anosov splitting for the geodesic flow on $T^1\HH^2$ is orthogonal with respect to the Sasaki metric.

\begin{prop}\label{prop_respectful}
Let $f\colon T^1\HH^2\to T^1\HH^2$ be the time-one map of the geodesic
flow on the hyperbolic plane and let $\hat g$ be the Sasaki metric on $\HH^2$. For each $n\ge 1$ denote by
$D\tilde\rho_n\colon T^1\HH^2\to T^1\HH^2$ a lift of $D\rho_n\colon T^1S\to T^1S$ with respect the locally isometric cover~(\ref{eq_covering}).
Then the sequence $\{ D\tilde\rho_n\colon T^1\HH^2\to T^1\HH^2; n\ge 1\}$ is ph-respectful relative to $(f,\hat g)$.
\end{prop}

We proceed with the proof of Theorem~\ref{theorem_basic} assuming the above proposition.

\begin{proof}[Proof of Theorem~\ref{theorem_basic}]
The lift of $D\rho_n\circ f_n$ with respect to the locally isometric cover of~(\ref{eq_covering}) is $D\tilde\rho_n\circ f$. Therefore it suffices to check absolute partial hyperbolicity of $D\tilde\rho_n\circ f$ with respect to $\hat g$. For sufficiently large $n$ this follows by combining Propositions~\ref{prop_respectful} and~\ref{prop_ph}.

Notice that diffeomorphism $f_n$ is clearly isotopic to the identity while  $\rho_n\colon S \to S$ is well known to be of infinite order in the mapping class group of $S$ (see \eg~\cite{FLP}).
Because the horizontal homomorphisms in the commutative diagram
\begin{equation*}
\xymatrix{
 \pi_1(T^1S) \ar[r]\ar_{\pi_1(D\rho_n)}[d] &  \pi_1(S)\ar^{\pi_1(\rho_n)}[d]\\
  \pi_1(T^1S)\ar[r] & \pi_1(S)
}
\end{equation*}
are epimorphisms we also have that $D\rho_n\colon S \to S$ has infinite order in the mapping class group.
Therefore $D\rho_n\circ f_n$ is of infinite order in the mapping class group of $T^1S$.
\end{proof}

\begin{proof}[Proof of Proposition~\ref{prop_respectful}]
Our strategy is to first establish properties~(\ref{eq_respectful1}) and~(\ref{eq_respectful2})
for the sequence $\{D\rho_n\colon (T^1S,\hat g_n)\to (T^1S,\hat g_n); n\ge 1\}$
and then deduce that~(\ref{eq_respectful1}) and~(\ref{eq_respectful2}) also hold for the lifts to $T^1\HH^2$.

Recall that by Corollary~\ref{cor_almost_isometry2} we already have
$$
d_{C^\infty}((D\rho_n)^*\hat g_n,\hat g_n)\to 0, n\to\infty.
$$
Because we take the lifts with respect to  locally isometric covers $(T^1\HH^2,\hat g)\to (T^1S,\hat g_n)$, it follows that
$$
d_{C^\infty}((D\tilde\rho_n)^*\hat g,\hat g)\to 0, n\to\infty.
$$

 On a hyperbolic surface, the Anosov splitting can be read off locally from the metric.
 Therefore $D\rho_n$ actually preserves the $Df_n$-invariant splitting $TT^1S=E_{f_n}^{ss}\oplus E_{f_n}^c\oplus E_{f_n}^{uu}$ outside  the neck.
 Further, on the neck, because all surfaces have the same universal cover, the splittings $TT^1S=E_{f_n}^{ss}\oplus E_{f_n}^c\oplus E_{f_n}^{uu}$ are
 uniformly continuous in $n$. Hence, Corollary~\ref{cor_close_to_id} yields
 $$
 \sup_{x\in M}\angle(D(D\rho_n)E_{f_n}^\sigma(x), E_{f_n}^\sigma(D\rho_n(x))\to 0, n\to \infty, \,\,\,\sigma=ss, c, uu.
 $$
 And because the splittings for $f_n$ lift to the splitting for $f$ we obtain
 $$
 \sup_{x\in M}\angle(D(D\tilde\rho_n)E_{f}^\sigma(x), E_{f}^\sigma(D\tilde\rho_n(x))\to 0, n\to \infty, \,\,\,\sigma=ss, c, uu,
 $$
which concludes the proof.
 \end{proof}

\subsection{The volume preserving modification}
\label{section_vp}
Denote by $m_n$ the Liouville volume on $(T^1S,g_n)$.

\begin{theo} \label{theorem_basic_vol}
Let $S$ be a surface of genus 2 or higher, let $f_n\colon T^1S\to T^1S$, $n\ge 1$ be the time-one maps of the geodesic flows and let $\rho_n\colon S\to S$, $n\ge 1$, be the Dehn twists as described earlier.  Then there exists a sequence of diffeomorphisms $\{h_n\colon T^1S\to T^1S; n\ge 1\}$ (which fiber over $\rho_n$) such that diffeomorphisms $h_n\circ f_n$ preserve $m_n$ and for all sufficiently large $n$ diffeomorphisms $h_n\circ f_n$ are absolutely partially hyperbolic. Furthermore, these diffeomorphisms and their finite iterates are not homotopic to identity.
\end{theo}

\begin{proof}
 Recall that the neck $C_n\subset (S,g_n)$ can be equipped with coordinates $(\bar x,\bar y)\in [0,1]\times S^1$ so that the expression for $g_n$~(\ref{eq_hyp_metric}) is independent of $n$. We identify $T^1C_n$ with $[0,1]\times S^1\times S^1$ and will use $\alpha$ for the last (angular) coordinate with the agreement that  vectors with $\alpha=0$ are tangent to geodesics $y=\textup{const}$.

 Locally the Liouville volume $m_n$ is the product of the Riemannian volume on $(S,g_n)$ and the angular measure on the tangent circle. A direct calculation in $(\bar x, \bar y, \alpha)$ coordinates yields the following formula for $m_n$ on $T^1C_n$
 \begin{equation}
 \label{eq_liouville}
 dm_n=\cosh(\bar x)(\cosh^{-1}(\bar x)\cos^2(\alpha)+\cosh(\bar x)\sin^2(\alpha))d\bar x d\bar y d\alpha
 \end{equation}

Now define $h_n\colon (T^1S,\hat g_n)\to (T^1S, \hat g_n)$ by
\begin{equation*}
 \label{eq_rho}
h_n(v)=
\begin{cases}
v, \mbox{if}\,\, v\notin T^1C_n \\
(\bar x, \bar y+\ell_n\rho(\bar x), \alpha), \mbox{if}\,\, v=(\bar x,\bar y,\alpha)\in T^1C_n
\end{cases}
\end{equation*}
Note that diffeomorphisms $h_n$, indeed, fiber over the Dehn twists $\rho_n$, and hence, $h_n\circ f_n$ and its finite iterates are not homotopic to identity for all $n\ge 1$.

The geodesic flows leave corresponding Liouville measures invariant. Thus, to show that $h_n\circ f_n$ preserves $m_n$ we have to check that $h_n$ preserves $m_n$. But $h_n$ preserves $\bar x$ and $\alpha$ coordinates and the expression for the density of $m_n$~(\ref{eq_liouville}) does not depend on $\bar y$. Hence, indeed, $h_n^*m_n=m_n$.

It remains to establish absolute partial hyperbolicity of $h_n\circ f_n$ for sufficiently large $n$. Just as in the proof of Theorem~\ref{theorem_basic}, we will check that a sequence of lifts $\{\tilde h_n\colon T^1\HH^2\to T^1\HH^2; n\ge 1\}$ (taken with respect to locally isometric covers) is a ph-respectful sequence relative to $(f,\hat g)$.

The restriction of the diffeomorphism $D\rho_n\colon T^1S\to T^1S$ to $T^1C_n$ is given by the formula
$$
D\rho_n(\bar x,\bar y, \alpha)=(\bar x, \bar y+\ell_n \rho(\bar x), D_{\bar x}(\alpha)),
$$
where $D_{\bar x}\colon S^1\to S^1$ is induced by the matrix
$$\left(\begin{matrix} 1&0 \\ \ell_n\rho'(\bar x)&1 \end{matrix}\right).$$
Therefore we can decompose $h_n$ as
$$
h_n=h_n'\circ D\rho_n,
$$
where the restriction of $h_n'$ to $T^1C_n$ is given by
$$
h_n'(\bar x,\bar y, \alpha)=(\bar x,\bar y, D^{-1}_{\bar x}(\alpha)).
$$
Since $\ell_n\to 0$ as $n\to\infty$ we have $d_{C^\infty}(h_n', id_{T^1S})\to 0$ as $n\to\infty$. Now we consider lifts $\tilde h_n'\colon T^1\HH^2\to T^1\HH^2$ which fiber over $id_{\HH^2}$ so that we also have $d_{C^\infty}(\tilde h_n', id_{T^1\HH^2})\to 0$.
It is easy to see, that this last conclusion together with ph-respectful properties of $\{D\tilde \rho_n; n\ge 1\}$ implies the sequence $\{\tilde h_n=\tilde h_n'\circ D\tilde \rho_n; n\ge 1\}$ is also ph-respectful relative to $(f,\hat g)$.
\end{proof}

\subsection{Stable ergodicity and robust transitivity}\label{sect:RTSE-T1S}

Here we establish an addendum to Theorem~\ref{theorem_basic_vol}, which finally completes our first proof of Theorem~\ref{t.main}.
\begin{add}
 The volume preserving absolutely partially hyperbolic diffeomorphism $F_0$ constructed in the proof of Theorem~\ref{theorem_basic_vol} admits a stably ergodic and robustly transitive perturbation $F$.
\end{add}

We do not know whether $F_0$ is ergodic or not with respect to the volume $m$ but it is possible to consider a small $C^1$ volume preserving perturbation $F_1$ of $F_0$ to make it stably ergodic (see \cite{BMVW}, in this context it is also possible to make a $C^\infty$-small perturbation to obtain ergodicity, \cite{HHU-densityergodicity}).

This means that there exists a $C^1$-neighborhood $\cU_1$ of $F_1$ such that for every $F \in \cU_1$ which preserves $m$ and is of class $C^{2}$ we have that $F$ is ergodic\footnote{The fact that $F$ has to be $C^2$ is for technical reasons which we shall not explain here.}. Moreover, every $F \in \cU_1$ is accessible (see \cite{BMVW}). Therefore, if $F$ preserves $m$ then it is transitive (even if it is not $C^2$).

However, in principle, it is possible that a dissipative perturbation of $F_1$ is not transitive. To obtain robust transitivity we shall perform yet another (volume preserving) perturbation of $F_1$ within $\cU_1$ to obtain both properties at the same time.

We have the following.

\begin{prop}\label{prop:volumepresperturbationRT} Let $f\colon M \to M$ be a volume preserving partially hyperbolic diffeomorphism of a 3-dimensional manifold such that it has a normally hyperbolic circle leaf whose dynamics is conjugate to a rotation. Then, there exists an arbitrarily $C^1$-small volume preserving perturbation of $f$ which makes it stably ergodic and robustly transitive.
\end{prop}

\begin{proof} The proof compiles several well known results. Let us fix $\eps>0$. As explained above, one can make a $\eps/3$-small $C^1$-volume preserving perturbation $f_1$ of $f$ such that $f_1$ is ergodic (in particular it is transitive). If $\eps$ is sufficiently small, by normal hyperbolicity of the circle, the circle persists and it is at $C^1$-distance smaller than $\eps/3$ of a rotation.

We would like to put ourselves in the hypothesis of \cite[Proposition 7.4]{BDV} which requires the construction of blenders. Notice first, that since $f$ is partially hyperbolic, the first hypothesis of the proposition is verified.

Notice that it was proved originally in \cite{BD} (and later in \cite{HHTU} in the conservative setting) that it is always possible to make a $\eps/3$-small $C^1$ (volume preserving) perturbation which creates a blender.

Using the transitivity and the normally hyperbolic circle leaf we know that the center-stable and center-unstable manifolds of the circle are robustly dense (see the proof of \cite[Proposition 7.4]{BDV}). Making a perturbation along the circle (of $C^1$-size less than $\eps/3$) one gets a Morse-Smale dynamics on the circle with only two periodic orbits (since it is close to a rotation) and by chosing their position one can guarantee that their strong manifolds intersect the activating region of the blender. This provides the second and third hypothesis of \cite[Proposition 7.4]{BDV} which provide the desired robust transitivity.

A further arbitrarily small $C^1$-perturbation gives stable ergodicity in addition to robust transitivity.
\end{proof}

The diffeomorphism $F_0$ is in the hypothesis of the Proposition  since there is at least one center circle leaf disjoint from all perturbations and since the dynamics on the circle is the time-1 map of a flow (and therefore conjugate to a rotation). This completes the proof of Theorem \ref{t.main} for this type of manifolds.

\subsection{Further remarks}
\subsubsection{Multiple Dehn twists}
\label{subsection_dehn_twists}
Note that given a collection of disjoint simple closed geodesics on a hyperbolic surface we can pick a sequence of hyperbolic metrics so that the lengths of all these geodesics go to zero. By considering disjoint collars we can perform Dehn twists along these geodesics simultaneously and, thus, obtain partially hyperbolic representative in corresponding mapping class.

% \subsubsection{Action on homology; pseudo-Anosov mapping class?} A difficult question which arises is to determine precisely which mapping classes of $T^1S$ have partially hyperbolic representatives. An a priori simpler question is to describe automorphisms of $H_1(T^1S;\RR)\simeq \RR^{2g}$ which are induced by partially hyperbolic diffeomorphisms (here $g$ is the genus of $S$). Note that for our Dehn twist construction the corresponding automorphism is either identity (when the geodesic bounds) or has an entry over the diagonal (when the geodesic is homologically non-trivial).
%

 \subsubsection{Higher dimensions} By Mostow rigidity, the mapping class group of a negatively curved manifold of dimension 3 or higher is finite. Therefore our scheme cannot be applied to a geodesic flow on such manifold to produce partially hyperbolic diffeomorphisms whose finite iterates are not homotopic to identity.

  \subsection{Modification on graph manifolds: examples based on Handel-Thurston Anosov flows.}

%  \marg{A: Can we claim that these are different from from examples in the next section? R: I believe so. But maybe it is some work. Maybe we can say that they provide graph manifolds and these might coincide with the others (even if the flows are different).}

The purpose of this section is to explain that our construction can be modified to yield partially hyperbolic diffeomorphism on graph manifold which admit Handel-Thurston Anosov flows~\cite{HT}. We merely observe that our mechanism for hyperbolicity is very well compatible with Handel-Thurston mechanism and, hence, they can be applied at the same time.

Let us briefly recall the Handel-Thurston construction. Let $S$ be a hyperbolic surface of genus two or higher and let $\gamma$ be a simple closed geodesic on $S$. Cutting along $\gamma$ creates two boundary components for $T^1S$ which are 2-tori $\TT^2=S^1\times S^1$, where the first $S^1$ corresponds to $\gamma$ and the second $S^1$ to the fiber circle. To a obtain the graph manifold $M$ reglue these boundary components with a shearing map $F\colon \TT^2\to \TT^2$ given by
$$
(x,\alpha)\mapsto (x+a\alpha,\alpha),
$$
where $a\in\ZZ\backslash\{0\}$. Because the angular coordinate stays unchanged, the differential $DF$ matches the Anosov vector field on
the boundary components. Hence the geodesic flow on $T^1S$ induces a flow on $M$.
The construction is summarized on the following figure taken from~\cite{HT}.
\begin{figure}[htbp]
\begin{center}
\includegraphics{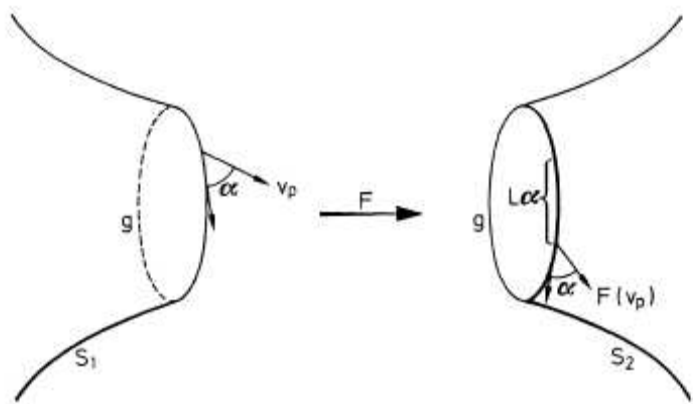}
\end{center}
 \caption{Shear with $a=1$}
\label{fig1}
\end{figure}
Among other things, Handel and Thurston showed that if one makes an appropriate choice of $a$, then this flow is a volume preserving Anosov flow.

Now, as before, we fix a sequence of hyperbolic metrics $\{g_n; n\ge 1\}$ on $S$ such that the length of $\gamma$ tends to zero. Each of these metrics yields a Handel-Thurston flow on the graph manifold $M$ whose time-one map is denoted by $f_n\colon M\to M$. Note that our construction of the Dehn twist occurs in the ``one-sided collar'' of $\gamma$. Hence the Dehn twists $\rho_n$, $n\ge 1$, induce diffeomorphism $D\rho_n\colon M\to M$ (which are no longer differentials, but rather ``glued differentials"; however we keep the same notation for consistency).

\begin{theo}
\label{theorem_basic3}
Let $M$, $f_n$ and $D\rho_n$, $n\ge 1$, be all as described above. Then for all sufficiently large $n$ the diffeomorphisms $D\rho_n\circ f_n$ are absolutely partially hyperbolic. Furthermore, these diffeomorphisms and their finite iterates are not homotopic to identity.%\marg{A: proof a reference to later section needed. R: I think that the proof applies in many contexts, but maybe not all, for example, if the surface is genus 2 and we do the dehn-twist and HT surgery along the separating (homologically trivial curve), the argument has to be adapted. But for many cases it works essentially the same.}
\end{theo}
\begin{proof}[Sketch of the proof of Theorem~\ref{theorem_basic3}]
The mechanism of hyperbolicity of Handel and Thurston is summarized on Figures~\ref{figHT2} and~\ref{figHT3} taken from~\cite{HT}.
\begin{figure}[htbp]
\begin{center}
\includegraphics{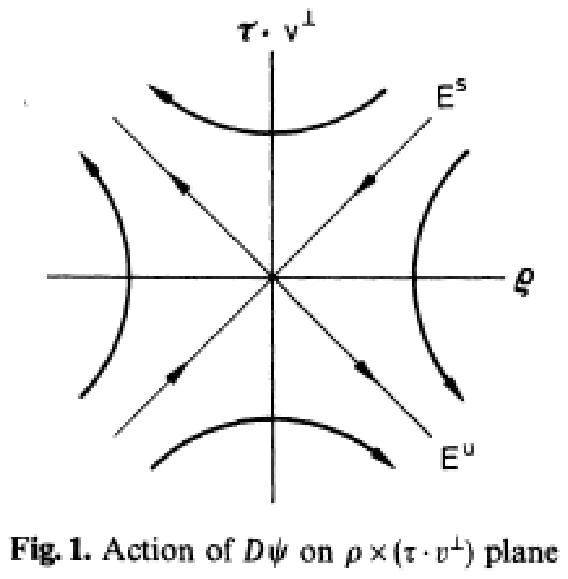}
\end{center}
 \caption{}
\label{figHT2}
\end{figure}
\begin{figure}[htbp]
\begin{center}
\includegraphics{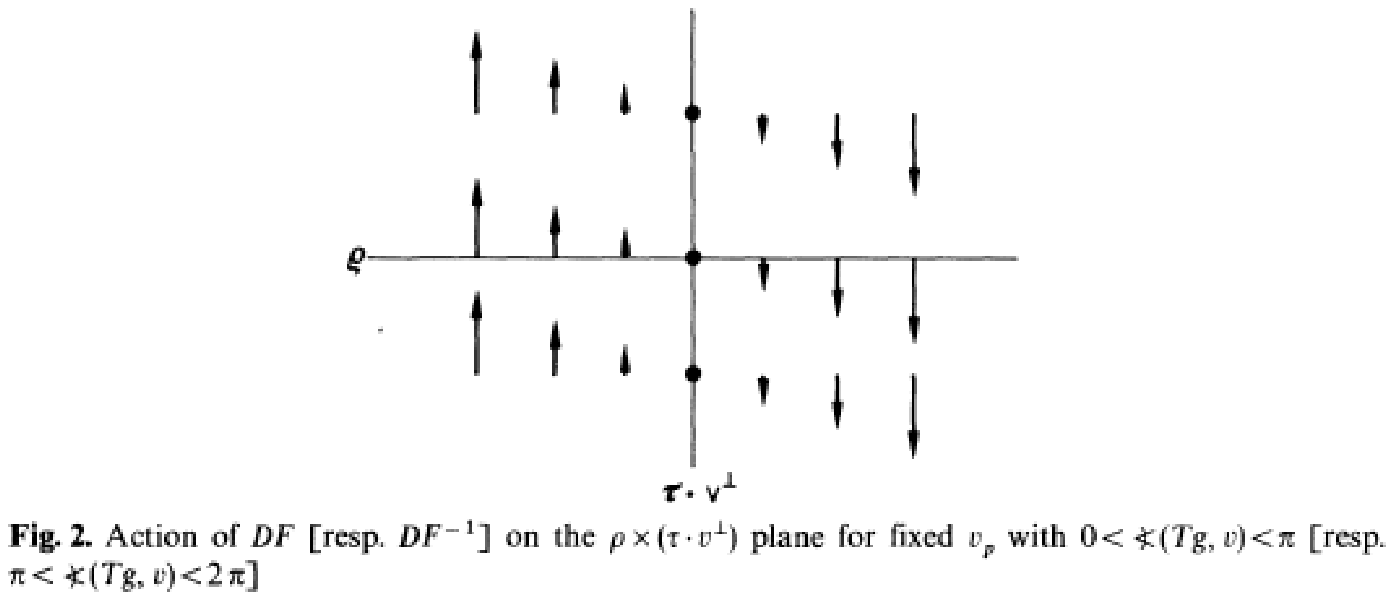}
\end{center}
 \caption{}
\label{figHT3}
\end{figure}
Figure~\ref{figHT2} depicts the action of the geodesic flow on the plane transverse to the flow in certain special coordinates. Figure~\ref{figHT3} depicts the action of the differential $DF$ of the gluing map on the transverse plane in the same coordinates. Here the strength of shear depends on $L$ --- the length of $\gamma$. Hence, as length of $\gamma$ tends to zero, $DF$ becomes close to identity. It readily follows that the Anosov splitting of the Handel-Thurston flow gets arbitrarily close to the original Anosov splitting. Hence the former is almost orthogonal. Further, the expansion and contraction rates are also remain almost the same. Therefore, by direct inspection one can check that our proof of Theorem~\ref{theorem_basic} goes through in this setting as well.

To finish the proof, one must show that the resulting
diffeomorphism and its iterates are not homotopic to identity. For
the sake of brevity, we will only indicate the needed results from
3-manifold topology:
\begin{itemize}
\item In \cite{Waldhausen} it is shown that for irreducible and
sufficiently large 3-manifolds (the manifolds we are dealing here
are irreducible since they admit an Anosov flow and sufficiently
large since they contain an incompressible tori) homotopy and
isotopy classes coincide, so it is enough to show that $D\rho_n$
is not isotopic to the identity.

\item In \cite[Proposition 4.1.1]{McC} it is shown that in this
situation (i.e. the manifold is not the mapping torus of a linear
Anosov diffeomorphism of $\TT^2$), if $D\rho_n$ were isotopic to
the identity then one could assume that the isotopy fixes the
torus on which we have cut the manifold all along the isotopy.

\item In \cite[Proposition 25.3]{Johannson} the mapping class
group of the resulting pieces after cutting along the tori is
studied. In particular, one can use this result and the previous
remark to check that $D\rho_n$ is not isotopic to identity.
\end{itemize}
\end{proof}

\begin{rema}
 The construction of a volume preserving modification of this example is the same as the one in Subsection~\ref{section_vp}. Existence of stably ergodic and robustly transitive perturbations can be seen in the same way as in Subsection~\ref{sect:RTSE-T1S}.
\end{rema}
\begin{rema}
Similarly to our remark in~\ref{subsection_dehn_twists}, we can also pick two collections of disjoint simple closed geodesic on $S$ (the collections may coincide) and perform Handel-Thurston surgery with respect to one set and the Dehn twist construction with respect to the second set.
\end{rema}
\begin{rema}\label{rema_HT_on_homology}
We also would like to remark that in the case when at least one of the Dehn twists is done on a non-separating geodesic along which the Handel-Thurston surgery is not performed then one can directly check that the induced homology automorphism is of infinite order. Hence, in this case, one does not need to rely on 3-manifold theory to see that the constructed diffeomorphism and its iterates are not homotopic to identity. In contrast, if the Handel-Thurston surgery and the Dehn twist are being done on the same geodesic (separating or not) the homology of this geodesic vanishes in the resulting graph manifold; and hence, the Dehn twist is identity in homology.
\end{rema}

%%%%%%%%%%%%%%%%%%%%%%%%%%%%%%%%%%%%%%%%%%%%%%%%%%%%%%%%%%%%%%%%%%%%%
%%%%%%%%%%%%%%%%%%%%%%%%%%%%%%%%%%%%%%%%%%%%%%%%%%%%%%%%%%%%%%%%%%%%%%
\section{An example on a graph manifold}\label{s.ExampleBL}
%%%%%%%%%%%%%%%%%%%%%%%%%%%%%%%%%%%%%%%%%%%%%%%%%%%%%%%%%%%%%%%%%%%%%%%%
%%%%%%%%%%%%%%%%%%%%%%%%%%%%%%%%%%%%%%%%%%%%%%%%%%%%%%%%%%%%%%%%

\subsection{Anosov flows transverse to tori}\label{ss.covering}  Let $Y$ be an Anosov vector field on a 3-manifold $M$ and let $T \subset M$ be a torus transverse to $Y$. It is
well known that $T$ must be incompressible~\cite{Br}. A systematic study of Anosov flows transverse to tori has been recently carried out in~\cite{BBY}, yet some questions still remain open. Most parts of our construction work well for all examples of Anosov flows transverse to tori which are not suspensions, however, at some stages we shall rely on the specific example from~\cite{BL} (particularly, Lemma \ref{l.foliationsinterT} below).

We will consistently use the same notation for vector fields and flows generated by them; \eg we write $Y^t$ for the flow generated by the vector field $Y$. We begin our presentation with the following lifting construction.

\begin{prop}\label{p.cycliccovering} For every $t_1>0$ there exists a finite connected covering $\hat M\to M$ such that if $\hat Y$ is the lift of $Y$ and $\hat T$ is a (connected component of) lift of $T$ then for all $t\in(0,t_1)$ one has $\hat
T \cap \hat Y^t(\hat T)= \emptyset$. Moreover, with respect to the metric on $\hat M$ induced by the covering map, the $C^1$-norm of $\hat Y$ is the same as the one of $Y$ and the time-one map $\hat Y^1$ is $(\ell, \lambda,\mu)$-partially hyperbolic whenever $Y^1$ is.
\end{prop}

Above we call a partially hyperbolic diffeomorphism $f: M \to M$ a {\it $(\ell, \lambda,\mu)$-partially hyperbolic} if

$$ \|Df^\ell |_{E^{ss}(x)}\| < \lambda < \|Df^\ell|_{E^c(x)} \| < \mu < \|Df^\ell |_{E^{uu}(x)}\| $$

\begin{proof} Consider the manifold $M_0$ obtained by cutting $M$ along $T$. If $M_0$ is not connected, then this means that $Y^t(T)$ is disjoint from $T$ for all $t \neq 0$ and therefore there is no need to consider a lift of $M$, \ie the posited property holds for $Y$.

If $M_0$ is connected, then it has two boundary components $T_1$ and $T_2$ such that $Y$ points inwards on $T_1$ and outwards on $T_2$. Let $t_0>0$ be the minimal time for an orbit to go from $T_1$ to $T_2$ so that for every $x \in T_1$ we have $Y^t(x) \notin T_2$ for $0\leq t \leq t_0$. Now glue $[t_1/t_0]+1$ copies of $M_0$ by identifying the copies of $T_2$ with the copies of $T_1$ and closing-up the last copy to obtain a compact boundaryless manifold $\hat M$ which covers $M$. Clearly, if $\hat T_1$ is a lift of  $T_1$ then $\hat Y^t (\hat T_1) \cap \hat T_1= \emptyset$ for $0 < t \leq t_1$.

The fact that the norm and the $(\ell, \lambda,\mu)$-partial hyperbolicity are not affected is direct from the fact that the differentiable and metric structures are obtained by lifting those from $M$.
\end{proof}

\begin{rema}\label{r.bundleslift} It is also easy to see that the bundles $E^{\sigma}_{\hat Y}$ are
the lifts of the bundles $E^{\sigma}_{Y}$ as well as the $\hat
Y_t$-invariant foliations ($\sigma=cs,cu,ss,uu$).
\end{rema}

\subsection{Coordinates in flow boxes}

As before, we consider an Anosov vector field $Y$ transverse to a torus $T$.  Since $Y$ is transverse to $T$ we obtain that the foliations $W^{cs}_{Y}$ and $W^{cu}_{Y}$ induce (transverse) foliations $\cL^{cs}_{Y}$ and $\cL^{cu}_{Y}$ on $T$.

We can consider coordinates $\theta_T\colon T \to \TT^2$ where $\TT^2= \RR^2 /_{\ZZ^2}$ with the
usual $(x,y)$-coordinates ($mod \ 1$). (The choice of $\theta_T$ is not canonical, and will be specified later in Lemma \ref{l.foliationsinterT}.)

We denote by $F^s$ and $F^u$ the foliations $\theta_T(\cL^{cs}_{Y})$ and $\theta_T(\cL^{cu}_{Y})$, respectively.

By Proposition \ref{p.cycliccovering}, for each $N>0$ and  integer $K\geq 1$ there exist a finite covering $\hat M_{N,K}\to M$ such that  $\hat T \cap \hat Y^t(\hat T)= \emptyset$ for $0 \leq t \leq NK$.  Then the set
$$\cU_N = \bigcup_{0 \leq t \leq N} \hat Y_t(\hat T)$$
 is injectively embedded in $\hat M_{N,K}$ and the first $K$-iterates by $\hat Y^N$ of $\cU_N$ have mutually disjoint interiors. (We have slightly abused the notation by ignoring the dependence of $\cU_N$ on $K$, but this will not cause any confusion.)

Consider the ``straightening diffeomorphism" $H_N: \cU_N \to [0,1] \times \TT^2$ given by

\begin{equation}\label{eq.HN}
 H_N(\hat Y_t(p)) = \left(\frac{t}{N}, \theta_T(p)\right), p\in\hat T
\end{equation}

For fixed $N$ and $K$ we shall denote by $\hat W^{\sigma}_N$ and $\hat E^{\sigma}_N$ the corresponding foliations and invariant bundles for the lift  $\hat Y$ ($\sigma=cs,cu,ss,uu$). (Again, we suppress dependance on $K$ to avoid overloading the notation.)

We also denote by $\cF^{uu}$ and $\cF^{ss}$  the one-dimensional foliations of $[0,1] \times \TT^2$ which in $\{t\} \times \TT^2$ coincide with the foliations $\{t\} \times F^u$ and $\{t\} \times F^s$ respectively.

\begin{lemm}\label{l.limit} Diffeomorphism $H_N$ have the following properties
\begin{itemize}
\item  $H_N(\hat W^{cs}_{N} \cap \cU_N) = [0,1] \times F^s\stackrel{\mathrm{def}}{=}
\cF^{cs}$

\item $H_N(\hat W^{cu}_{N} \cap \cU_N) = [0,1] \times F^u \stackrel{\mathrm{def}}{=}
\cF^{cu}$

\item $DH_N(\hat E^{ss}_{N})$ converges to the tangent
bundle of the foliation $\cF^{ss}$ as $N \to
\infty$.

\item $DH_N(\hat E^{uu}_{N})$ converges to the  tangent
bundle of the foliation $\cF^{uu}$ as $N \to
\infty$.

\end{itemize}
\end{lemm}

It is important to remark that the above convergence is with respect to the standard metric in $[0,1] \times \TT^2$ and not with respect to the push forward metric from the manifold via $H_N$.

\begin{proof} Because the differential of $H_N$ maps the vector field $\hat Y$ to the vector field $\frac{1}{N} \frac{\partial}{\partial t}$ the first two properties follow. Note that the component of $\hat E^{\sigma}_{N}$ along $\hat Y$ is uniformly bounded
($\sigma=ss,uu$). Therefore, contraction by a factor $\frac{1}{N}$ implies the posited limit behavior in the latter properties.
\end{proof}

\subsection{A diffeomorphism in a flow box which preserves transversalities}

Assume that there exists a smooth path $\{\varphi_s\}_{s\in [0,1]}$ of diffeomorphisms
of $\TT^2$ such that

\begin{itemize}
\item $\varphi_s=Id$ for $s$ in neighborhoods of $0$ and $1$,
\item the closed path $s\mapsto \varphi_s$ is not homotopically
trivial in $\diff(\TT^2)$, \item for every $s \in [0,1]$:
$$\varphi_s(F^u)\pitchfork F^s.$$
\end{itemize}

We use the coordinate chart $H_N\colon \cU_N\to[0,1]\times\TT^2$ to define
 diffeomorphism $\cG_N : \cU_N \to \cU_N$ by
$$
(s,x,y) \mapsto (s, \varphi_s(x,y)).
$$

The following lemma is immediate from our choice of $\{\varphi_s\}_{s\in [0,1]}$.
\begin{lemm}\label{l.perturbationorbitspace}
The diffeomorphism $\cG_N$ has the following properties:
\begin{itemize}
\item $\cG_N(H_N^{-1}(\cF^{uu}))$ is
transverse to $\hat W^{cs}_{\hat Y} = H_N^{-1}(\cF^{cs})$,
\item $\cG_N(\hat W^{cu}_{\hat Y})=\cG_N(H_N^{-1}(\cF^{cu}))$ is transverse
to $H^{-1}_N(\cF^{ss})$.
\end{itemize}
\end{lemm}

Let $h_N \colon \hat M_{N,K} \to \hat M_{N,K}$ be the diffeomorphism which coincides with $\cG_N$ on $\cU_N$ and is identity outside $\cU_N$.

\begin{coro}\label{c.transversalityBL} There exists $N_0>0$ such that for $N \geq N_0$:
\begin{itemize}
\item $Dh_N(E^{uu}_{\hat Y})$ is transverse to $E^{cs}_{\hat Y}$,
\item $Dh_N(E^{cu}_{\hat Y})$ is transverse to $E^{ss}_{\hat Y}$.
\end{itemize}
\end{coro}

\begin{proof} This follows by combining Lemma \ref{l.perturbationorbitspace}, Lemma~\ref{l.limit} and the fact that outside $\cU_N$ the diffeomorphism $h_N$ is the identity.
\end{proof}

\begin{rema}\label{r.volumepreserving}
Notice that if there is a volume form $\omega$ on $T$ such that $\varphi_s$ preserves the form $(\theta_T)_\ast (\omega)$ for every $s$, then $h_N$ preserves volume form $\omega \wedge dY$.
\end{rema}

\subsection{Proof of partial hyperbolicity}

Let diffeomorphism $f_{N,K} \colon \hat M_{N,K} \to \hat M_{N,K}$ be the time-$N$ map of the flow generated by the vector field $\hat Y$ on $\hat M_{N,K}$.

\begin{prop}\label{p.phofBL}
For any sufficiently large $N$ there exists $K_0=K_0(N)$ such that for all $K\geq K_0$ the diffeomorphism $h_N \circ f_{N,K}$ is absolutely partially hyperbolic.
\end{prop}

\begin{proof} Let $F_{N,K}= h_N \circ f_{N,K}$. The proof uses the cone-field criteria.

Recall (see Proposition \ref{p.cycliccovering}) that for large enough $N$ and any $K$ we have that $f_{N,K}: \hat M_{N,K} \to \hat M_{N,K}$ is $(1, \lambda, \mu)$-partially hyperbolic for some $\lambda<1<\mu$. We shall choose $1< \mu_1 < \mu_0 < \mu$ (and $\lambda< \lambda_0 < \lambda_1 < 1$ for the symmetric argument).

First consider a fixed value of $N\geq N_0$ given by Corollary \ref{c.transversalityBL}. Then, using partial hyperbolicity and Corollary \ref{c.transversalityBL}, we can choose a cone-field $\cE^{uu}$ about $E^{uu}_{\hat Y}$ such that

\begin{itemize}
\item cone-field $\cE^{uu}$ is transverse to $E^{cs}_{\hat Y}$;
\item $Df_{N,K}(\overline{\cE^{uu}})\en \cE^{uu}$;
\item cone-field $Dh_N (\cE^{uu})$ is also transverse to $E^{cs}_{\hat Y}$;
\end{itemize}
and
\begin{itemize}
\item for every $v \in \cE^{uu}\setminus \{0\}$ one has that $\|Df_{N,K} v \| > \mu_0 \|v\|$;
\item for every $v \in T_x\hat M_{N,K}$ such that $Df_{N,K} v \notin \cE^{uu}$ one has that $\|Df_{N,K} v\| \leq \mu_1 \|v\|$.
\end{itemize}

We shall show that if $K$ is sufficiently large then we can construct a cone-field $\cC^{uu}$ such that for a sufficiently large iterate $n_K>0$ we have

$$
DF_{N,K}^{n_K} (\overline{\cC^{uu}}) \en \cC^{uu}
$$
and there exists $\hat \mu>1$ such that

$$
\|DF_{N,K}^{n_K}v \| > \hat \mu \|v \|,\;\; \mbox{if}\;\;v\in \cC^{uu}\setminus \{0\}.
$$
and
$$
\|DF_{N,K}^{n_K} v\| \leq \hat \mu \|v\|,\;\;\mbox{if}\;\; DF_{N,K}^{n_K} v \notin \cC^{uu}.
$$

Consider the minimal $n_0$ such that $Df_{N,K}^{n_0}(Dh_N(\overline{\cE^{uu}})) \en \cE^{uu}$. Because the angle between $Dh_N(\cE^{uu})$ and $E^{cs}_{\hat Y}$ is uniformly bounded from below,  the value of $n_0$ is independent of $K$. Define the cone-field $\cC^{uu}$ as follows:

\begin{itemize}
\item $\cC^{uu}=Dh_N(\cE^{uu})$ on $\cU_N$,
\item $\cC^{uu}=\cE^{uu}$ outside $\cU_N$.
\end{itemize}

\begin{clai}
There exists $n>0$ such that $DF_{N,K}^n (\overline{\cC^{uu}}) \en \cC^{uu}$.
\end{clai}

\begin{proof}
Notice that $\cC^{uu}$ is indeed smooth since $h_N$ coincides with the identity in a neighborhood of the boundary of $\cU_N$. Let us first show that if $n = 2\ell n_0$ with $\ell \geq 2$ and $K\gg 2n_0$ then
$$
DF_{N,K}^n (\overline{\cC^{uu}}) \en \cC^{uu}.
$$
This is quite direct. Notice first that for points $x \notin \cU_N \cup f_{N,M}^{-1}(\cU_N)$ one has that $DF_{N,K} (\overline{\cC^{uu}}(x)) \en \cC^{uu}(f_{N,M}(x))$ since this holds true for $\cE^{uu}$ and $f_{N,M}$  (and $F_{N,K}= f_{N,M}$ on $f_{N,M}^{-1}(\cU_N)$). For points $x \in f_{N,M}^{-1}(\cU_N)$ we have that $DF_{N,K} (\overline{\cC^{uu}(x)}) = Dh_N Df_{N,K}(\overline{\cC^{uu}(x)}) \en Dh_N (\cE^{uu}(f_{N,K}(x))) = \cC^{uu}(F_{N,K}(x))$. Finally, if $x \in \cU_N$ then, by construction, we have that $DF_{N,K}^{n_0}(\overline {\cC^{uu}}(x)) \en \cC^{uu}(F_{N,K}(x))$. This implies the inclusion

$$ DF_{N,K}^{2n_0} (\overline{\cC^{uu}}) \en \cC^{uu}. $$
\end{proof}

To show absolute partial hyperbolicity (\ie the existence of $\hat \mu$ as stipulated earlier) it is enough to consider large enough $\ell$ as above and $n_K=2\ell n_0$ (and $K \geq n_K$, recall that $n_0$ and $\ell$ are independent of $K$). Let $C = \max_x \{\|DF_{N,K}(x)\|,\|DF_{N,K}^{-1}(x)\|\}$.

Notice first that for any point $x\in M_{N,K}$ one has that $F^i_{N,K}(x)\in \cU_N$ for at most one value of $0 \leq i \leq n_K$. This implies that:

\begin{itemize}
\item on the one hand, if $v \in \cC^{uu}$, then $\|DF_{N,K}^{n_k} v\| \geq C^{-1} \mu_0^{n_k-1} \|v\|$;
\item on the other hand, if $DF_{N,K}^{n_k} v \notin \cC^{uu}$ then $\|DF_{N,K}^{n_k} v\| \leq C \mu_1^{n_k-1} $.
\end{itemize}

\noindent which for large enough $n_k$ verifies the desired properties.

Finally, by reversing the time, a symmetric argument provides a cone-field $\cC^{ss}$ with analogous properties and, therefore, yields absolute partial hyperbolicity of $F_{N,K}$ for a sufficiently large $K$.
\end{proof}

\subsubsection{Some remarks on this approach} 
The approach presented here requires to consider large finite coverings of the initial manifold in order to ensure large return times to the transverse tori. This method provides uniform bounds on the constants of the partial hyperbolicity and can be compared with the mechanism used in section~\ref{s.Exampleunittangent} to construct examples starting from geodesic flows in constant negative curvature. 

After finishing the first draft of this paper, we discovered a different mechanism which guarantees partial hyperbolicity after composing with Dehn twists and allows one to avoid passing to finite covers. This mechanism is related to the one here yet involves some different ideas. The current construction may actually suit better when trying to understand the dynamics of new examples. 

We will leave to it a future paper to explore this different mechanism which will also unify the mechanisms for both (families of) examples presented in this paper (the ones with an incompressible torus transverse to the flow and the ones with an incompressible torus not transverse to the flow).

\subsection{Bonatti-Langevin's example}\label{ss.BL}

Notice that to this point we do not know if there is an Anosov flow transverse to a torus for which the posited family
of diffeomorphisms $\{\varphi_s\}_{s\in [0,1]}$ exists. For this purpose we shall introduce a specific class of Anosov flow examples
from~\cite{BL} which will also make  easy the task of showing volume preservation. It is plausible that other examples, \eg those that can be found in \cite{BBY}, also can serve as the Anosov flow ingredient in our construction. However we haven't checked it.

A relevant remark is that Proposition \ref{p.phofBL} can be applied to the suspension of a linear Anosov diffeomorphism of $\TT^2$
which gives rise to a manifold where every partially hyperbolic is leaf conjugate to an Anosov flow (\cite{HP}).
It is important to consider a manifold on which the diffeomorphism $h_N$ is not isotopic
to the identity in order to obtain a new example of partially hyperbolic dynamics.

A volume preserving transitive Anosov flow $X_0^t\colon M_0\to M_0$ was built in \cite{BL}
which admits a torus $T_0$ transverse to the flow. At the same time, the flow admits a
periodic orbit disjoint from $T_0$. We denote by $X_0$ the vector field generating this flow.

We will state in the following lemma the properties about this flow that we shall
use to construct a family $\varphi_s$ as in the previous subsection:

\begin{lemm}\label{l.foliationsinterT} There exist coordinates $\theta: T_0 \to \TT^2$ such that:

\begin{itemize}
\item In the coordinates $(x,y)$ of $\TT^2$ the foliation $F^{s}$
is a foliation with two horizontal circles and such that every
leaf is locally of the form $(x,g^s(x))$, where $g^s$ is a function with
derivative smaller than $\frac 14$.

\item In the coordinates $(x,y)$ of $\TT^2$ the foliation $F^{u}$
is a foliation with two vertical circles and such that every leaf
is locally of the form $(g^u(y),y)$, where $g^u$ is a function with
derivative smaller than $\frac 14$.

\item The flow is orthogonal to $T_0$ and if $\omega$ is the area form induced in
$T_0$ by the invariant volume form one has that $\theta_\ast(\omega)$ is the standard area form $dx \wedge dy$ on $\TT^2$.
\end{itemize}
\end{lemm}

\begin{proof}
This follows by inspection of the construction in \cite{BL}.

The first two properties follow directly from the computation of the holonomy between the entry torus and the exit torus performed in page 639 of~\cite{BL}.
The relation between $h_t$ and the intersection of the center-stable and center-unstable foliations with $T_0$
can be inferred from the arguments at the end of the proof of the main theorem (see the end of page 642 and page 643 of \cite{BL}).

Preservation of the volume form in this coordinates is provided by \cite[Lemma 3.1]{BL}.
\end{proof}

\begin{figure}[htbp]
\begin{center}
\includegraphics{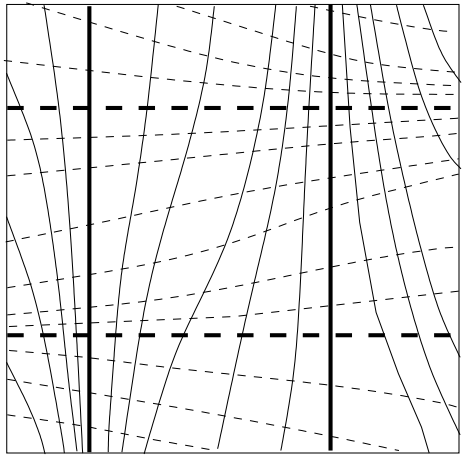}
\end{center}
 \caption{The foliations $F^s$ and $F^u$ in $\TT^2$ for the Bonatti-Langevin's example.}
\label{figHT3}
\end{figure}

Notice that for $X^t_0$ one can easily construct a family of diffeomorphisms $\{\varphi_s\}_{s \in [0,1]}$
consisting of translations in $\TT^2$ (which preserve the Lebesgue measure and transversality between the bundles)
and verifies the properties of the previous section.

\begin{coro}\label{c.constructionwithBL}
There exists a finite cover $M$ of $M_0$ and a number $N>0$ such that if we denote by $f$ the time-$N$ map
of the flow $X^t$ (the lift of $X^t_0$ to $M$) and a volume preserving diffeomorphism
$h_N \colon M \to M$ which is a Dehn twist along a lift $T$ of $T_0$ such that
the diffeomorphism $F= h_N \circ f$ is a conservative absolutely partially
hyperbolic diffeomorphism.
\end{coro}

\begin{proof} It follows directly from the previous lemma that there exists a family
$\{\varphi_s\}_{s\in [0,1]}$ with the properties required by Proposition \ref{p.phofBL}.
This gives the desired statement (see Remark \ref{r.volumepreserving} for the volume preservation).
\end{proof}

%%%%%%%%%%%%%%%%%%%%%%%%%%%%%%%%%%%%%%%%%%%%%%%%%%%%%%%%%%%%%%%%%%%%%%%%%%%5
\subsection{The iterates of the Dehn twist are not isotopic to the identity}\label{s.homology}
According to \cite{Br}, every torus transverse to an Anosov flow is incompressible
(\ie its fundamental group is injected in the fundamental group of the ambient manifold). Furthermore, in the case of the \cite{BL}-example, the manifold obtained by cutting along the torus is a circle bundle, and the gluing map along the torus does not preserve the homology class (in the torus) of the fibers:
in the terminology of the topology of $3$-manifolds, this means that the torus belongs to the Jaco-Shalen-Johannson (JSJ) family of the ambient manifold. A result of Johannson~\cite[Proposition 25.3]{Johannson} (see also~\cite[Proposition 4.1.1]{McC}) provides a general criterion for proving that a Dehn twist along a given torus of the JSJ family, and its iterates, are not isotopic to the identity.  This criterium indeed applies to the Dehn twist along the transverse torus of the example of~\cite{BL} (see also \cite{BarbotBL}).

Nevertheless, there is a very elementary proof for the specific example given in Section~\ref{ss.BL}. The goal of this section is to present this simple proof.

In Section~\ref{ss.covering}, for any manifold $M$ endowed with a transitive Anosov flow $X$ and a torus $T$ transverse to $X$,  for any integer $n>0$, we build a  covering obtained by considering $n$ copies of the manifold $M_0$ with boundary obtained by cutting $M$ along $T$.
Let us denote by $\Pi_n\colon \tilde M_n\to M$ this $n$-cyclic covering.  Note that $\Pi_n^{-1}(T)$ is a family of $n$ disjoint tori $\hat  T_0,\dots, \hat T_{n-1}$.
In this section we will denote by $V$ the underlying manifold of \cite{BL}-example and by $\Pi_n\colon V_n\to V$ the $n$-cyclic covering defined above.

\begin{prop}\label{p.BL-isotopy} Let $n$ be an integer divisible by 4 and let the manifold $V$, the transverse torus $T$ and the covering
$\Pi_n\colon  V_n\to V$ be as described above. Let $\hat T_0,\dots, \hat T_{n-1}$ be the lifts of $T$ to $V_n$.
Then for any non-zero homology class $[\gamma]$  of $H_1(\hat T_0,\ZZ)$, the Dehn twist along $\hat T_0$ in the direction of $[\gamma]$ and its positive iterates are not homotopic to the identity map on $V_{n}$.
\end{prop}

\subsubsection{Description of the ambient manifold $V$ and its $2$-cover}\label{ss.2cover}

Let $V_0$ be the compact manifold with boundary obtained by cutting $V$ along $T$.  By construction in~\cite{BL} we have the following facts:
\begin{itemize}
\item  the boundary of $V_0$ consists of two tori $T_0$ and $T_1$;
 \item $V_0$  is the total space of a circle bundle $S^1\to V_0\stackrel{p}{\to} B$;
 \item the base $B$ is the projective plane $\RR P^2$ with the interiors of two disjoint disks removed, or equivalently the M\oe bius
 band with the interior of a disk removed;
 \item the structure group of $p\colon V_0\to B$ contains orientation-reversing diffeomorphisms and $p\colon V_0\to B$ is the unique circle  bundle over $B$ for which the total space $V_0$ is oriented.
\end{itemize}
We denote by $\varphi\colon T_0\to T_1$ the orientation reversing gluing diffeomorphism (that is, $V=V_0/\varphi $).

We proceed with an explicit description of $p\colon V_0\to B$. Let
$$\tilde B= \RR/\ZZ \times [-1,1]\setminus \left(D\left((0,0),\frac12\right) \cup D\left(\big(\frac12,0\big),\frac12\right)\right)$$
where $D((t,0),\frac12)$ is the open disk of radius $\frac12$ centered at the point $(t,0)\in\RR/\ZZ\times [-1,1]$.
Let $$\tilde V_0=\tilde B\times \RR/\ZZ.$$
Note that $B$ is diffeomorphic to the quotient of $\tilde B$ by the involution without fixed points $(r,s)\mapsto (r+\frac12,-s)$ and $V_0$ is the quotient of
$\tilde V_0$ by the free involution $(r,s,t)\mapsto (r+\frac12,-s,-t)$.
We denote by $\pi_0\colon \tilde V_0\to V_0$ this 2-fold cover.

\begin{figure}[htbp]
\begin{center}
\includegraphics{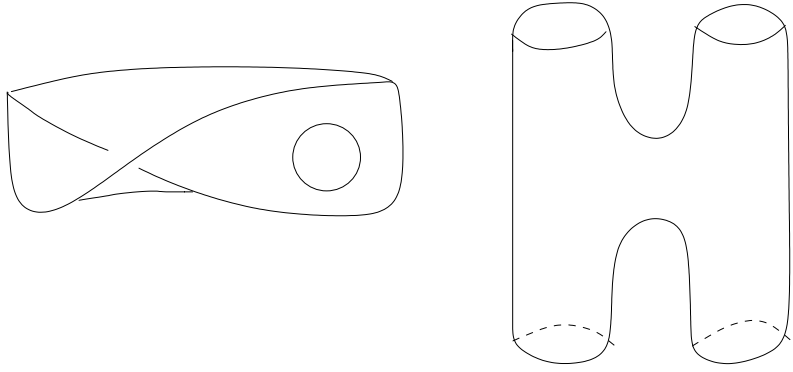}
\end{center}
 \caption{The base space $B$ (to the left) and its 2-fold covering $\tilde B$ (to the right).}
\label{figbandV0}
\end{figure}

Note that $\pi_0$ restricts to a diffeomorphism on each connected component of $\partial \tilde V_0$ (to a
connected component of $\partial V_0$). Let us denote
$$
\pi_0^{-1}(T_0)=T_{0,0}\sqcup T_{0,1} \mbox{ and }\pi_0^{-1}(T_1)=T_{1,0}\sqcup T_{1,1}.
$$
Let $\tilde \varphi_{0}\colon T_{0,0}\to T_{1,0}$ and $\tilde \varphi_{1}\colon T_{0,1}\to T_{1,1}$ be
the unique diffeomorphisms which project by $\pi_0$ to $\varphi$.
We denote by $\tilde \varphi\colon T_{0,0}\cup T_{1,0}\to T_{0,1}\cup T_{1,1}$ the assembled diffeomorphism which coincides with $\tilde \varphi_{i}$ on $T_{i,0}$, $i=0,1$.

Let
$$\tilde V  = \tilde V_0/ \tilde \varphi.$$
One can easily check that $\pi_0$  induces a covering map $\pi\colon \tilde V \to V$.

\subsubsection{Description of the gluing map $\varphi$}\label{ss-gluing}
The manifold $\tilde V_0$ is a product of $\tilde B$ and the circle $S^1$.  The horizontal $2$-foliation
(whose leaves are the $\tilde B\times \{t\}$) and the vertical circle bundle
 passes to the quotient by the involution on $V_0$.  They induce on each connected component
 $T_0,T_1,T_{0,0},T_{0,1},T_{1,0},T_{1,1}$ two transverse foliations by circles
  called respectively meridians (induced by the horizontal $2$-foliation) and parallels (or fibers).

According to \cite{BL} the gluing map $\varphi$ is chosen so that it maps the parallels and meridians of $T_0$ to meridians
and parallels of $T_1$, respectively. Thus  $\tilde \varphi$ exchanges parallels with meridians in the same way (for other constructions see \cite{BarbotBL}).

We will need the following remark.

\begin{rema}\label{r.+} Fix an orientation of $V_0$. It induces orientations of $T_0$ and of $T_1$ (as boundary orientation).  We orient $\tilde V_0$ so that the projection
$\pi_0$ preserves the orientation.  Now $T_{i,j}$ inherit of the boundary orientation and $\pi_0\colon T_{i,0}\cup T_{i,1}\to T_i$ is orientation preserving.

Up to changing all orientations, the gluing map $\varphi$ is a quarter of turn in the positive direction, that is,
(in coordinates putting the meridians and the fibers in horizontal and  vertical position)  a rotation by $+\frac\pi4$.

Because $\pi$ is orientation preserving, both restrictions of $\tilde \varphi$ to $T_{0,0}$ and $T_{0,1}$ (endowed with the boundary orientations) are rotations by $+\frac\pi4$.
\end{rema}

\subsubsection{$4m$-cyclic covers of $V$ and $\tilde V$; the proof of Proposition~\ref{p.BL-isotopy}}\label{ss.4mcover}

Recall that the cyclic covering $\Pi_n\colon V_n\to V$ is obtained by considering $n$ copies  $V_0\times \{i\}$, $i\in\ZZ/n\ZZ$, of $V_0$
and by gluing $T_0\times \{i\}$ with $T_1\times \{i+1\}$ using  $\varphi$.  In the same way be define a cyclic covering
$\tilde \Pi_n\colon \tilde V_n\to \tilde V$  obtained by considering $n$ copies  $\tilde V_0\times \{i\}$, $i\in\ZZ/n\ZZ$, of $\tilde V_0$
and by gluing $(T_{0,0}\cup T_{0,1})\times \{i\}$ with $(T_{1,0}\cup T_{1,1})\times \{i+1\}$ using  $\tilde \varphi$.
It is easily to check that $\pi_0$  induces a covering map $\pi_n\colon \tilde V_n \to V_n$ which projects to $\pi$.

Proposition \ref{p.BL-isotopy} is a corollary of the next lemma.

\begin{lemm}\label{l.homology} For any $n\in 4\NN\setminus\{0\}$,
the homomorphism $H_1(T_{i,j}\times\{k\},\ZZ)\to H_1(\tilde V_n)$,
$i,j\in\{0,1\}$, $k\in\ZZ/n\ZZ$ induced by the inclusion
$T_{i,j}\times\{k\}\subset \tilde V_n$ is injective.
\end{lemm}

\begin{proof}[Proof of Proposition~\ref{p.BL-isotopy}] The Dehn twist $h_\gamma$ on $V_n$ along $T_0\times\{i\}$ in the direction of $\gamma\in H_1(T_0,\ZZ)$
can be lifted to a diffeomorphisms $\tilde h_\gamma\colon\tilde V_n\to\tilde V_n$  which is the composition of two Dehn twists with disjoint supports, one along $T_{0,0}\times\{i\}$ and one along $T_{0,1}\times \{i\}$ in the direction of the lifts of $\gamma$.  For proving that $h_\gamma^k$, $k\in\ZZ\setminus \{0\}$, is not homotopic to the identity map of $V_n$, it is enough to prove that $\tilde h_\gamma^k$ is not homotopic to the identity on $\tilde V_n$.

Notice that, by construction,  $T_{0,0}\times\{i\} \cup T_{0,1}\times \{i\}$ does not disconnect $\tilde V_n$.
This implies that there is a closed path $\sigma$ which intersects  $T_{0,0}\times\{i\} \cup T_{0,1}\times \{i\}$ transversely at a single point contained in $T_{0,0}\times\{i\}$.
In particular, the homology class of $\sigma$ in $H_1(\tilde V_n,\ZZ)$ is non-trivial.  Let $[\gamma_0]$ be the homology class of the lift of $\gamma$ to $T_{0,0}\times\{i\}$. Lemma~\ref{l.homology}
 implies that the  class $[k.\gamma_0]$ in $H_1(\tilde V_n,\ZZ)$ does not vanish when $k\neq 0$.  Clearly
 $$\left(\tilde h_\gamma^k\right)_*([\sigma])=[\sigma]+k[\gamma_0]\neq [\sigma].$$
 Thus the action of $\tilde h^k_\gamma$ on homology is non-trivial, completing the proof.
\end{proof}

\subsubsection{Oriented surfaces in $\tilde V_4$ and the proof of Lemma~\ref{l.homology}}

\begin{lemm}There is a smooth oriented closed surface $\Sigma_0\subset \tilde V_4$ which intersects transversally each circle fiber of $\tilde V_0\times\{0\}$ at exactly one point.
 \end{lemm}
\begin{proof} Let  $S= \tilde B\times \{0\}\subset \tilde V_0$. It is a surface with boundary which intersects each circle fiber in exactly one point.

We consider $S\times\{0\}\subset \tilde V_0\times\{0\}\subset \tilde V_4$ and $ S\times\{2\}\subset \tilde V_4$.  Consider the intersection of these surfaces
with $\tilde V_0\times \{1\}$:
\begin{itemize}
\item $S\times\{0\}\cap \tilde V_0\times \{1\}$ is the image by $\tilde \varphi$ of  $(S\cap (T_{0,0}\cup T_{0,1}))\times \{0\}$.
In other words, it consists of exactly one circle fiber  in $T_{1,0}\times \{1\}$ and one circle fiber in $T_{1,1}\times \{1\}$.
 \item in the same way $S\times\{1\}\cap \tilde V_0\times \{1\}$ is the image by $\tilde \varphi^{-1}$ of  $(S\cap (T_{1,0}\cup T_{1,1}))\times \{2\}$ that is,
 it consists of exactly one circle fiber  in $T_{0,0}\times \{1\}$ and one circle fiber in $T_{0,1}\times \{1\}$.
\end{itemize}
There are two disjoint cylinders $C_0\times\{1\}$ and $C_1\times\{1\}$, each being a
$1$-parameter family of circle fibers in $\tilde V_0\times\{1\}$ such that:
\begin{itemize}
\item the cylinder $C_0\times\{1\}$ connects the fiber in $T_{1,0}$ with the one in $T_{0,0}$ and
\item the cylinder $C_1\times \{1\}$ connects the fiber in $T_{1,1}$ with the one in $T_{0,1}$.
\end{itemize}

The surface we obtain still has four boundary components which are circle fibers in $\tilde V_0\times \{3\}$. And each
 boundary components of $\tilde V_0\times \{3\}$ contains precisely one of these circles.  We obtain the posited surface $\Sigma_0$ by gluing in the cylinders $C_0\times \{3\}$ and $C_1\times \{3\}$.

\begin{figure}[htbp]
\begin{center}
\includegraphics{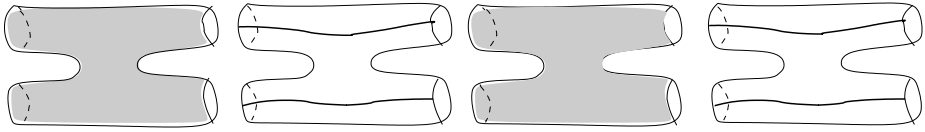}
\end{center}
 \caption{{\small Construction of the surface $\Sigma_0$. It consists of the shaded surfaces (corresponding to leaves  of the horizontal foliation) and the product of the curves connecting the boundaries multiplied by the fibers in the parts where the surfaces are not shaded.}}
\label{fig4cover}
\end{figure}

 To finish the proof it remains to see that $\Sigma_0$ is orientable. For this we need to pay attention to the orientations of the fibers glued by the cylinders.
 More precisely, orientability follows from our description of the gluing map and the following properties:
 \begin{itemize}
  \item the orientations of the circle fibers in $T_{1,0}\times \{1\}$ and in $T_{1,1}\times \{1\}$, (viewed as the image by $\tilde \varphi$ of  boundary components
 of the surface $S\times\{0\}$) are the same (as circle fibers in $V_0\times\{1\})$.
 \item the orientations of the circle fibers in $T_{0,0}\times \{1\}$ and in $T_{0,1}\times \{1\}$, (viewed as the image by $\tilde\varphi^{-1}$  of boundary components
 of the surface $S\times\{3\}$) are the same (as circle fibers in $V_0\times\{s\})$.
 \end{itemize}
 But these follow from Remark~\ref{r.+}, completing the proof of the lemma.
\end{proof}

Recall that $\tilde V_4$ is a cyclic $4$-cover of $\tilde V_0$. Let $\rho$ be a generator of the Deck transformation group (sending $\tilde V_0\times \{i\}$ on
$\tilde V_0\times \{i+1\}$).  Let $\Si_1=\rho(\Sigma_0)$.
We can now conclude the proof of Lemma~\ref{l.homology} and therefore of Proposition~\ref{p.BL-isotopy}.

\begin{lemm}Any non-trivial $[\gamma]\in H_1(T_{i,j}\times \{k\},\ZZ)$ has a non-zero intersection  number with either $\Sigma_0$ or $\Sigma_1$.
\end{lemm}

For the proof notice that $(T_{i,j}\times \{k\})\cap (\Sigma_0\cup \Sigma_1)$ is the union of exactly one meridian and one parallel.

%%%%%%%%%%%%%%%%%%%%%%%%%%%%%%%%%%%%%%%%%%%%%%%%%%%%%%%%%%%%%%%%%%%%%%%%%%%%%
\subsection{Stable ergodicity and robust transitivity}\label{ss.RobustTransBL}

The same argument as in Subsection \ref{sect:RTSE-T1S} shows that in this setting one also obtains stable ergodicity and robust transitivity by
a small $C^1$ volume preserving perturbation. 

The argument only uses the existence of a circle leaf where the dynamics is a rotation; in this example, as Bonatti-Langevin's example has a periodic orbit which is disjoint from the transverse tori, this orbit is not altered by the modifications one has made and therefore, (any of the lifts) of this orbit remains a closed normally hyperbolic center curve whose dynamics is conjugate to a rotation and it can be used for the argument. 

This completes the proof of Theorem \ref{t.main} for this family of manifolds.

%%%%%%%%%%%%%%%%%%%%%%%%%%%%%%%%%%%%%%%%%%%%%%%%%%%%%%%%%%%%%%%%%%%%%%%%%%%%%%%%%%%%%%%%%%%%%%%%%%%%%%%
\subsection{Periodic center leaves}\label{s.saddlenode}

In \cite[Chapter 7]{BDV} the classification of transitive partially hyperbolic systems in dimension 3 is also discussed.
Problem 7.26 is related to the structure of the center leaves of a transitive partially hyperbolic diffeomorphisms.
It asks whether it is possible to have periodic points both in circle leaves and in leaves homeomorphic to the line.
It also asks whether it is possible that a transitive partially hyperbolic diffeomorphism has a periodic
center leaf homeomorphic to $\RR$ for which one end is contracting and the other one is expanding
(such leaves are called \emph{saddle node leaves}; see \cite[Section 7.3.4]{BDV}).

We shall show that some of our examples 
admit both periodic center leaves which are lines as well as periodic center leaves  which are circles. The existence of complete curves tangent to $E^c$ invariant under some iterate of $f$ through periodic points is established in \cite{BDU,HHU-centerleafs}. 
%
%The existence of a fixed circle leaf follows from the fact that the Anosov flow had a fixed circle
%leaves disjoint from  the support of the large perturbation we have performed (the Dehn twist); these circles are normally hyperbolic
%and therefore persist after all further  small perturbations. Thus it remains to prove that some of our examples admit periodic points which are not contained
%in periodic center circles.

\begin{prop}\label{p.centerleaf} There exists a volume preserving, robustly transitive absolutely partially hyperbolic diffeomorphism $F$ on a
closed $3$-manifold, which has periodic points on normally hyperbolic invariant circles
and periodic points which do not belong to any periodic invariant circle tangent to the center bundle.
\end{prop}

\begin{proof} The diffeomorphism is obtained following the procedure which we described in this section, starting with a transitive Anosov
flow  $X$ with a transverse torus, then considering a large finite
cyclic cover and composing the time $N$ map (for some $N$ large enough) and a
Dehn twist $h_N$ along one component of the lift of the transverse torus.

Specifically, recall notation of Section~\ref{s.homology} : let $V$ be
the closed $3$ manifold endowed with the of Anosov flow $X$ which admits a transverse torus. The map $\pi\colon\tilde V\to V$ is the
$2$-fold cover defined in Section~\ref{ss.2cover}.
Also recall that $\tilde V_4$ and, more generally, $\tilde V_{4m}$ are the cyclic covers of $\tilde V$ defined in Section~\ref{ss.4mcover}. Let
$\tilde X_{4m}$ denote the lift of $X$ on $\tilde V_{4m}$.   The Anosov flow $\tilde X_{4m}$ admits at least $8m$ non-isotopic transverse tori denoted by
$T_{i,j}\times\{k\}$, $i,j\in\{0,1\}$, $k\in \ZZ/4m\ZZ$ following the  notation of Section~\ref{ss.4mcover}.
Here we will denote by $\tilde T_{4m}$ the torus
$T_{0,0}\times \{0\}\subset \tilde V_{4m}$.

Our starting Anosov flow is $(\tilde V_4,\tilde X_4)$, with the transverse torus $\tilde T_0$.
Notice that $(\tilde V_{4m}, \tilde X_{4m})$ is a cyclic cover of $(\tilde V_4,\tilde X_4)$ which respects dynamics. The torus $\tilde T_{4m}$ is one  of the connected components of the
lift of $\tilde T_0$ (the projection from $\tilde T_{4m}$ to $\tilde T_0$ is a diffeomorphism), and the return time of $\tilde X_{4m}$ to $\tilde T_{4m}$ tends
to $+\infty$ as $m\to+\infty$.  We can repeat the construction we have performed in Section~\ref{ss.BL}: for a sufficiently large $N>0$ there exist $m$
such that the composition of the time $N$-map of the flow of $\tilde X_{4m}$ and the Dehn twist $h_N$ along $\tilde T_{4m}$ (in the direction
$[\gamma]\in H_1(\tilde T_{4m},\ZZ)$)  is partially hyperbolic and volume preserving. Further,
 an extra small perturbation yields a robustly transitive stably ergodic diffeomorphisms $f$.

Now, as a consequence of Lemma~\ref{l.homology} (arguing exactly as in the proof of Proposition~\ref{p.BL-isotopy}) we obtain: given any closed curve
$C\subset \tilde V_{4m}$ whose homological intersection number $i(C)$ with $\tilde T_{4m}$ does not vanish, and for any $n\neq 0$ we have
$$[f^n(C)]=[C]+i(C)\cdot[\gamma]\neq [C]\in H_1(\tilde T_{4m},\ZZ).$$
Hence, if $C$ is a periodic circle for $f$ then its intersection number $i(C)$ vanishes.

\begin{lemm}\label{l.transverse} For sufficiently large $N$ and $m$, the center bundle of $f$ is orientable and transverse to $\tilde T_{4m}$.
\end{lemm}

Now we finish the proof of Proposition~\ref{p.centerleaf} assuming the above lemma. Because center bundle is oriented and transverse to $\tilde T_{4m}$,
any closed center curve passing in a neighborhood of $\tilde T_{4m}$ has a non-trivial intersection number with $\tilde T_{4m}$,
and therefore cannot be periodic. On the other hand, because $f$ is robustly transitive, up to performing an arbitrarily small $C^1$-perturbation, it admits
hyperbolic periodic points in the neighborhood of $\tilde T_{4m}$.

In \cite{BDU} the existence of complete invariant center curves through these periodic points is shown (see also \cite{HHU-centerleafs}).
These curves are periodic and intersect $\tilde T_{4m}$, therefore, they cannot be circles.

Finally, the support of the Dehn twist $h_N$ is disjoint from a hyperbolic basic set of
the vector field $\tilde X_{4m}$ and this basic set contains periodic orbits
of $\tilde X_{4m}$.  The restriction of $f$ to this basic set is a small perturbation of the time-$N$ map of the flow $X_{4m}$. Hence these invariant circles
persists as normally hyperbolic invariant circles. Up to performing an arbitrarily small perturbation, one may assume that each of these circles
contains a hyperbolic periodic points. Hence, we have shown that $f$ has both: periodic points on invariant center circles
and periodic points not in invariant center circles, concluding the proof.

 It remains to establish Lemma~\ref{l.transverse}.
\end{proof}

%\marginpar{A: maybe explain/ recall what is $h_T$, here. Also the argument with normally hyperbolic circles seem to appear twice. The first time before the statement of Proposition. Maybe ok.}

\begin{proof}[Proof of Lemma~\ref{l.transverse}]
Let $E^{cs}_X$ $E^{cs}_f$, $E^{cu}_X$ $E^{cu}_f$ be the center stable and unstable bundles of the vector field $\tilde X_{4m}$ and of $f$, respectively.
When $m$ tends to $\infty$ the time return of $\tilde X_{4m}$ to $\tilde T_{4m}$ also tends to infinity. As a consequence we obtain that,
on the fundamental domain $\cU_N$, the bundle $E^{cs}_f$ tends to $E^{cs}_X$ and $E^{cu}_f$ tends to $Dh_N(E^{cu}_X)$.  Therefore $E^c_f$ tends to
$E^{cs}_X\cap Dh_N(E^{cu}_X)$ which is everywhere transverse to the $\TT^2$ fibers in the fundamental domain $\cU_N$, concluding.
 \end{proof}

We expect that our examples should posses saddle node center leaves, we state this as a question here.

\begin{quest}
Does $f$ admit a periodic saddle node leaf?
\end{quest}

\medskip

{\noindent\textit{Acknowledgments.} The second author would like to thank Dmitry Scheglov for many enlightening discussions
and Anton Petrunin for a very useful communication. The third author thanks the hospitality of
the Institut de Mat\'ematiques de Bourgogne, Dijon.}

\end{document}

%% file: fig-longneck.pstex_t
\begin{picture}(0,0)%
\includegraphics{fig-longneck.pstex}%
\end{picture}%
\setlength{\unitlength}{3236sp}%
\begingroup\makeatletter\ifx\SetFigFont\undefined%
\gdef\SetFigFont#1#2#3#4#5{%
  \reset@font\fontsize{#1}{#2pt}%
  \fontfamily{#3}\fontseries{#4}\fontshape{#5}%
  \selectfont}%
\fi\endgroup%
\begin{picture}(7386,2299)(78,-2201)
\put(2921,-900){\makebox(0,0)[lb]{\smash{{\SetFigFont{12}{14.4}{\rmdefault}{\mddefault}{\updefault}$\gamma_n$}}}}
\put(3456,-1434){\makebox(0,0)[lb]{\smash{{\SetFigFont{12}{14.4}{\rmdefault}{\mddefault}{\updefault}$C_n$}}}}
\end{picture}%